\documentclass[11pt]{amsart}
\usepackage[margin=1in]{geometry}

\usepackage{amssymb}
\usepackage{amsthm}
\usepackage{amsmath}
\usepackage{mathrsfs}
\usepackage{amsbsy}
\usepackage{bm}
\usepackage{hyperref}
\usepackage{tikz}
\usetikzlibrary{calc}
\usepackage{array}
\usepackage{enumerate}
\usepackage{xcolor}
\usepackage{bbm}
\usepackage{comment}
\usepackage{mathtools}
\usepackage{microtype}

\definecolor{DangerousAxiomBlue}{RGB}{46,56,255}
\definecolor{DarkAxiomBlue}{RGB}{107,131,255}
\definecolor{AxiomBlue}{RGB}{153,187,255}
\definecolor{SafetyAxiomBlue}{RGB}{194,214,255}
\definecolor{DarkBlue}{RGB}{0,0,230}

\hypersetup{colorlinks=true, citecolor=AxiomBlue, linkcolor=DangerousAxiomBlue,urlcolor=DangerousAxiomBlue}

\allowdisplaybreaks
\usepackage[noadjust]{cite}
\usepackage{caption}
\usepackage[noabbrev,capitalise,nameinlink]{cleveref}
\crefname{conjecture}{Conjecture}{Conjectures}

\newtheorem{theorem}{Theorem}[section]
\newtheorem{proposition}[theorem]{Proposition}

\newtheorem{question}[theorem]{Question}

\newtheorem{lemma}[theorem]{Lemma}

\newtheorem{maintheorem}{Theorem}

\crefname{maintheorem}{Theorem}{Theorems}

\theoremstyle{definition}

\newtheorem{remark}[theorem]{Remark}

\newcommand{\SSS}{\mathfrak{S}}
\newcommand{\ZZ}{\mathbb{Z}}
\newcommand{\RR}{\mathbb{R}}
\newcommand{\E}{\mathbb{E}}
\newcommand{\Prob}{\mathbb{P}}
\newcommand{\1}{\mathbf{1}}
\DeclareMathOperator{\Ent}{Ent}
\DeclareMathOperator{\Var}{Var}
\DeclareMathOperator{\id}{id}
\numberwithin{equation}{section}

\newcommand{\dfn}[1]{\textcolor{DarkAxiomBlue}{\emph{#1}}}

\begin{document}

\title{Cutoff for the Adjacent Transposition Shuffle on a Cycle}

\author[]{Colin Defant}
\address[]{Axiom Math, 124 University Avenue, Palo Alto, CA 94301, USA}
\email{colin@axiommath.ai} 

\begin{abstract}
In the \emph{adjacent transposition shuffle on a cycle}, $n$ distinct cards are placed at the vertices of a cycle, and adjacent cards swap positions according to independent Poisson clocks of rate 1. We prove that this Markov chain exhibits total variation cutoff at time $n^2\log n/(8\pi^2)$, with a window of order at most $n^2\log\log n$.    
\end{abstract}

\maketitle

\section{Introduction}\label{sec:intro}
\subsection{Shuffling}
Card shuffling provides some of the most concrete examples of finite-state Markov chains. Given a procedure for repeatedly modifying a deck of cards, one can ask
how many steps are required before the resulting ordering is nearly
indistinguishable from a uniformly random permutation. Questions of this kind
have played an important role in the development of the theory of mixing times
for Markov chains. The answer depends dramatically on what moves are allowed. For example, the
classical Gilbert--Shannon--Reeds model of riffle shuffling, analyzed
by Bayer and Diaconis \cite{BayerDiaconis}, mixes a deck of $n$ cards after
roughly $\frac32\log_2 n$ shuffles. Another fundamental model allows two
arbitrary cards to exchange positions; Diaconis and Shahshahani
\cite{DS} showed that the resulting random-transposition
shuffle exhibits an abrupt transition to equilibrium on the scale
$\frac12 n\log n$. These examples helped establish card shuffling as a
particularly useful testing ground for ideas concerning convergence to
equilibrium and the cutoff phenomenon.

A rather different situation arises when the dynamics are constrained by
geometry. Suppose that cards can exchange positions only when they are
adjacent in some graph. Information can then propagate only locally, so mixing is governed by diffusion rather than by the essentially mean-field behavior of random transpositions. This simple restriction connects card shuffling with models from interacting particle systems and statistical mechanics. Indeed, if one forgets some or all of the card labels, the same graphical dynamics gives variants of the symmetric exclusion process. 

The most extensively studied example of this type is the \dfn{adjacent
transposition shuffle on a path}. Aldous \cite{Aldous} obtained the first
general bounds on its mixing time. Wilson \cite{Wilson} later determined the
correct order of the mixing time and obtained sharp bounds on its leading
constant. Lacoin \cite{LacoinSegment} subsequently proved that it exhibits cutoff at time $n^2\log n/(2\pi^2)$. Closely related questions for the symmetric
exclusion process have also been studied on the cycle; in particular, Lacoin
\cite{LacoinCircle} proved cutoff with a diffusive window for exclusion on the
cycle.

In this article, we consider the corresponding shuffle of \emph{labeled} cards on a cycle. At first glance, replacing a path by a cycle may appear to be a minor modification. However, this modification breaks the monotonicity that was crucial for the analysis of the path. Overcoming this obstacle has been a serious challenge. Our main result does this, showing that the adjacent transposition shuffle on a cycle exhibits cutoff and
determining the cutoff time.

\subsection{Mixing Times and Cutoff}
The \dfn{total variation distance} between two probability measures $\alpha,\beta$ on a finite set $\Omega$ is 
\[\|\alpha-\beta\|_{\mathrm{TV}}=\frac12\sum_{z\in\Omega}|\alpha(z)-\beta(z)|=\max_{A\subseteq\Omega}|\alpha(A)-\beta(A)|. \]
Suppose $\mathcal X=(X_t)_{t\geq 0}$ is an irreducible continuous-time Markov chain with state space $\Omega$ whose stationary distribution is $\pi$. We write $P_t^{\mathcal X}(\sigma,\sigma')=\mathbb P(X_t=\sigma'\mid X_0=\sigma)$ for the probability the chain is in state $\sigma'$ given it started at state $\sigma$. We view $P_t^{\mathcal X}(\sigma,\cdot)$ as a probability distribution on $\Omega$ and let \[d_n^{\mathcal X}(t)=\max_{\sigma\in\Omega}\|P_t^{\mathcal X}(\sigma,\cdot)-\pi\|_{\mathrm{TV}}.\] For $0<\epsilon<1$, the \dfn{mixing time} of $\mathcal X$ is the quantity \[t_{\mathrm{mix}}^{\mathcal X}(\epsilon)=\inf\{t\geq 0:d_n^{\mathcal X}(t)\leq \epsilon\}.\] 

One of the central themes of the modern theory of Markov chains is the cutoff phenomenon, which concerns a sequence $(\mathcal X^{(N)})_{N\geq 1}$ of finite-state Markov chains \cite{AldousDiaconis, DiaconisCutoff, LevinPeres}. We say this sequence exhibits \dfn{cutoff} if there is a sequence $(t_N)_{N\geq 1}$ such that 
\[\lim_{N\to\infty}\frac{t_{\mathrm{mix}}^{\mathcal{X}^{(N)}}(\epsilon)}{t_N}=1\]
for every fixed $\epsilon\in(0,1)$. In this situation, it is common to simply say that $\mathcal X^{(N)}$ exhibits cutoff at time $t_N$. 

\subsection{Setup and Main Result}

Let us view the symmetric group $\mathfrak S_n$ as the group of bijections from $\ZZ/n\ZZ$ to itself. Let $G=(\mathbb Z/n\mathbb Z,E)$ be a simple graph. A permutation $\sigma\in \mathfrak S_n$ can be seen as an arrangement of cards on the vertices of $G$ where $\sigma(i)$ is the position of card $i$. The \dfn{interchange process} on $G$ is a continuous-time Markov chain with state space $\mathfrak S_n$. The edges of $G$ are assigned independent Poisson clocks of rate $1$; when the clock assigned to an edge $e$ rings, the cards on the endpoints of $e$ swap places. (More generally, an interchange process may assign clocks of different rates to the edges, but we assume here that all clocks have rate $1$.)  

Throughout this article, we fix $n\geq 3$ and let $\mathcal{C}_n=(\ZZ/n\ZZ,E(\mathcal C_n))$ be the cycle graph with edge set $E(\mathcal{C}_n)=\{\{x,x+1\}:x\in \ZZ/n\ZZ\}$. For $x\in \ZZ/n\ZZ$, let $\tau_x$ be the transposition that swaps $x$ and $x+1$. The \dfn{adjacent transposition shuffle on the cycle} is the interchange process on $\mathcal C_n$. Equivalently, it is the continuous-time Markov chain with generator
\begin{equation}\label{eq:shuffle-generator}
\mathcal L_n f(\sigma)=\sum_{x\in \ZZ/n\ZZ}\bigl(f(\tau_x\circ\sigma)-f(\sigma)\bigr).
\end{equation}
This chain is irreducible, and it is reversible with respect to the uniform probability measure $\pi_n$ on $\SSS_n$. Write $P_t(\sigma,\cdot)$ for its distribution at time $t$ when it starts at $\sigma$, and let $\mu_t=P_t(\id,\cdot)$, where $\id$ is the identity permutation. Let
\[d_n(t)=\max_{\sigma\in\SSS_n}\|P_t(\sigma,\cdot)-\pi_n\|_{\mathrm{TV}}=\|\mu_t-\pi_n\|_{\mathrm{TV}},\] where the second equality comes from the symmetry arising from the action of $\mathfrak S_n$. Let
\[t_{\mathrm{mix}}^{(n)}(\varepsilon)=\inf\{t\geq 0:d_n(t)\leq\varepsilon\}. \] 

To state our main theorem, we introduce the notation 
\begin{equation}\label{eq:cutoff-time}
\lambda_n=2-2\cos(2\pi/n)\quad\text{and}\quad t_n=\frac{\log n}{2\lambda_n}.
\end{equation}

\begin{maintheorem}\label{thm:main}
There is an absolute constant $C>0$ such that, for every $\varepsilon\in(0,1)$, there is a constant $C_\varepsilon>0$ satisfying
\[
t_n-C_\varepsilon n^2\leq t_{\mathrm{mix}}^{(n)}(\varepsilon)\leq t_n+Cn^2\log\log n+C_\varepsilon n^2
\]
for all sufficiently large $n$. In particular, $t_{\mathrm{mix}}^{(n)}(\varepsilon)\sim n^2\log n/(8\pi^2)$.
\end{maintheorem}

\cref{thm:main} tells us that the adjacent transposition shuffle on the cycle $\mathcal C_n$ exhibits cutoff with the cutoff time $n^2\log n/(8\pi^2)$.

\subsection{Related Work}
Starting with the interchange process on a graph $G$, one can obtain a simpler process by choosing some $1\leq k \leq n-1$ and keeping track of only the vertices occupied by the cards $1,\ldots,k$. The resulting \dfn{simple exclusion process} exchanges occupied and unoccupied neighboring vertices at rate 1. Morris \cite{Morris} determined the correct order of the mixing time on discrete tori using an auxiliary color-changing process called the \emph{chameleon process}. Oliveira \cite{Oliveira} developed this approach further and bounded mixing times of simple exclusion processes on arbitrary weighted graphs in terms of the mixing times of the associated 1-particle walks.

For the simple exclusion process on the cycle $\mathcal C_n$, Lacoin \cite{LacoinCircle} identified the cutoff time $n^2\log k/(8\pi^2)$ when $k\leq n/2$ and $k\to\infty$, and he proved that the cutoff window has order $n^2$. In a companion article, he identified the \emph{cutoff profile}, which provides even more refined information than the cutoff phenomenon alone \cite{LacoinProfile}. However, Lacoin explicitly noted that his cycle argument did not prove cutoff for the adjacent transposition shuffle on a cycle \cite[Section~1.4]{LacoinCircle}.

The relationship between the 1-particle simple exclusion process and the full interchange process is especially striking at the level of spectral gaps. Caputo, Liggett, and Richthammer \cite{CLR} proved Aldous' spectral gap conjecture, which states that on every finite connected weighted graph, the interchange process and the corresponding single-particle random walk have the same spectral gap. A spectral gap does not by itself give the leading constant in total variation mixing. One must also understand the size of the initial deviation and its equilibrium fluctuations. For the adjacent transposition shuffle on $\mathcal C_n$, we consider the statistic $F\colon\SSS_n\to\mathbb R$ defined by
\[F(\sigma)=\sum_{i\in \ZZ/n\ZZ}\cos\bigl(2\pi(\sigma(i)-i)/n\bigr).\]
It starts at $n$, its expectation at time $t$ is $ne^{-\lambda_n t}$, and its equilibrium standard deviation has order $\sqrt n$ (see \cref{sec:finish}). Equating the two scales suggests that $ne^{-\lambda_n t}=\sqrt n$, which gives exactly $t=t_n$. The eigenfunction method of Wilson \cite{Wilson} makes this intuition into a lower bound. The upper bound must rule out additional information in the joint configuration that survives substantially longer than this statistic.

This perspective also explains the factor of 4 between the cutoff times for the adjacent transposition shuffles on the path and on the cycle. The lowest positive eigenvalues of the two 1-card walks are asymptotically $\pi^2/n^2$ and $4\pi^2/n^2$, respectively. Comparing 1-particle and many-particle mixing is a useful idea in other contexts as well. The comparison method of Diaconis and Saloff-Coste \cite{DSC} transfers estimates between generators by comparing their Dirichlet forms. Alon and Kozma \cite{AlonKozma} developed operator comparisons using the octopus inequality, while Hermon and Salez \cite{HermonSalez} obtained sharp bounds on mixing times for high-dimensional products, including the hypercube. 

\begin{remark}
There is also a different model called the \emph{cyclic adjacent transposition shuffle}. Nam and Nestoridi \cite{NamNestoridi} proved cutoff for this systematic-scan shuffle, which updates successive edges of a \emph{path} rather than choosing independent random edges. The word ``cyclic'' in that model's name refers to the update schedule, not to a geometric cycle. 
\end{remark}

\subsection{Proof Strategy and Outline}

Our lower bound on $t_{\mathrm{mix}}^{(n)}(\epsilon)$ uses the slowly decaying statistic $F$ discussed above. The main work is to prove the matching upper bound. Our approach is to reveal the card positions in a uniformly random order and estimate, at each stage, how much information remains about the position of the next card. We obtain sufficiently strong control by averaging these errors over time and space. This leads to a polylogarithmic bound on the amount of information remaining shortly after $t_n$, after which a standard functional inequality brings the chain close to equilibrium in an additional $O(n^2\log\log n)$ time.

\Cref{sec:prelim} collects the preliminary estimates. In \cref{sec:exposure,sec:replica}, we develop the argument in which we randomly expose cards, and we study an auxiliary process used to control the resulting errors. \Cref{sec:local-gap,sec:blocks} establish the main estimates needed for this argument. We complete the proof of \cref{thm:main} in \cref{sec:finish} and discuss further questions in \cref{sec:future}.

\section{Entropy and Heat-Kernel Estimates}\label{sec:prelim}

Let $\nu$ be a probability measure on a finite set $\Omega$. For a function $g\colon\Omega\to\mathbb R_{\geq 0}$, define
\[\Ent_\nu(g)=\E_\nu[g\log g]-\E_\nu[g]\log\E_\nu[g], \]
where $0\log0=0$. If $\alpha$ is another probability measure on $\Omega$, its \dfn{relative entropy} with respect to $\nu$ is
\[H(\alpha\mid\nu)=\sum_{z\in\Omega}\alpha(z)\log\frac{\alpha(z)}{\nu(z)}. \]
If $\alpha$ assigns positive mass to a point to which $\nu$ assigns mass $0$, then $H(\alpha\mid\nu)=\infty$. Otherwise,
\[H(\alpha\mid\nu)=\Ent_\nu\left(\frac{\mathrm{d}\alpha}{\mathrm{d}\nu}\right). \]
We will use Pinsker's inequality
\[\|\alpha-\nu\|_{\mathrm{TV}}^2\leq\frac12H(\alpha\mid\nu). \]

Now let $L$ be the generator of a continuous-time Markov chain on a finite set $\Omega$. Write $q(z,z')$ for the transition rate from $z$ to $z'$ so that
\[Lf(z)=\sum_{z'\in\Omega}q(z,z')\bigl(f(z')-f(z)\bigr). \]
Suppose the chain is reversible with respect to a probability measure $\nu$. Its \dfn{Dirichlet form} is
\[
\mathcal D(f)=\langle f,-Lf\rangle_\nu
=
\frac12\sum_{z,z'\in\Omega}\nu(z)q(z,z')
\bigl(f(z')-f(z)\bigr)^2,
\]
where $\langle f,g\rangle_\nu=\E_\nu[fg]$. We write
\[
\Var_\nu(f)=\E_\nu\left[\left(f-\E_\nu[f]\right)^2\right].
\]
A \dfn{Poincar\'e inequality} with constant $A$ is an inequality
\[
\Var_\nu(f)\leq A\mathcal D(f)
\]
that holds for every real-valued function $f$ on $\Omega$. A \dfn{logarithmic Sobolev inequality} with constant $A$ is an inequality
\[
\Ent_\nu(f^2)\leq A\mathcal D(f)
\]
that holds for every real-valued function $f$ on $\Omega$.

\subsection{A Logarithmic Sobolev Inequality}

For an integer $N\geq 1$, let $\mathcal P_N$ be the path with vertex set $[N]:=\{1,\ldots,N\}$ and edge set
\[E(\mathcal P_N)=\{\{j,j+1\}:1\leq j\leq N-1\}. \]
For $1\leq i<j\leq N$, let $\tau_{i,j}^{(N)}$ denote the transposition in $\SSS_N$ that swaps the vertices $i$ and $j$, and write $\tau_j^{(N)}=\tau_{j,j+1}^{(N)}$. The \dfn{adjacent transposition shuffle on a path} is the continuous-time Markov chain on $\SSS_N$ with generator
\[L_N^{\mathrm{path}}f(\sigma)=\sum_{j=1}^{N-1}\bigl(f(\tau_j^{(N)}\circ\sigma)-f(\sigma)\bigr). \]
Equivalently, each edge of $\mathcal P_N$ has an independent Poisson clock of rate 1, and the cards at its endpoints exchange positions whenever that clock rings. This chain is reversible with respect to the uniform probability measure $\pi_N$ on $\SSS_N$, so $\pi_N$ is stationary for this chain. Let
\[\mathcal D_N^{\mathrm{path}}(f)=\langle f,-L_N^{\mathrm{path}}f\rangle_{\pi_N} \]
be its Dirichlet form. 

\begin{lemma}\label[lemma]{lem:lsi}
There is an absolute constant $c_1>0$ such that, for every $N\geq1$ and every function $f\colon\SSS_N\to\mathbb R$,
\[
\Ent_{\pi_N}(f^2)\leq c_1N^2\mathcal D_N^{\mathrm{path}}(f).
\]
\end{lemma}

\begin{proof}
The case $N=1$ is immediate. Suppose $N\geq2$, and let $k=\lfloor N/2\rfloor$. We use the balanced Bernoulli--Laplace inequality of Lee and Yau \cite[Theorem~5]{LeeYau}; see also \cite[Theorem~1.6]{Salez}. Let $\nu$ be the uniform probability measure on the set $\binom{[N]}{k}$ of $k$-element subsets of $[N]$. For every function $g\colon\binom{[N]}{k}\to\mathbb R_{\geq 0}$, we have 
\begin{equation}\label{eq:bl}
\Ent_\nu(g^2)\leq\frac{c_2}{N}\E_\nu\left[\sum_{\substack{a\in S\\b\notin S}}\bigl(g(S\setminus\{a\}\cup\{b\})-g(S)\bigr)^2\right].
\end{equation}
The sum on the right-hand side is not normalized.

Let $I_1=\{1,\ldots,k\}$ and $I_2=\{k+1,\ldots,N\}$. For a uniformly random permutation $\sigma\in\SSS_N$, let
\[S=\{a\in[N]:\sigma(a)\in I_1\}. \]
Then $S$ is uniformly distributed on $\binom{[N]}{k}$. Conditional on the event $S=A$, the labels in $A$ are uniformly arranged on $I_1$, the labels in $[N]\setminus A$ are uniformly arranged on $I_2$, and these two arrangements are independent. Define $g\colon\binom{[N]}{k}\to\mathbb R_{\geq 0}$ by 
\[g(A)=\left(\E_{\pi_N}[f^2\mid S=A]\right)^{1/2}. \]
The entropy chain rule tells us that 
\[\Ent_{\pi_N}(f^2)=\E_{A\sim\nu}\left[\Ent_{\pi_N(\,\cdot\,\mid S=A)}(f^2)\right]+\Ent_\nu(g^2).\] 

We first bound $\Ent_\nu(g^2)$. Fix $a\in S$ and $b\notin S$. If $\sigma$ has conditional law $\pi_N(\,\cdot\,\mid S)$, then exchanging the positions occupied by the cards $a$ and $b$ produces a permutation with conditional law
\[
\pi_N\bigl(\,\cdot\,\bigm|S\setminus\{a\}\cup\{b\}\bigr).
\]
Thus, the reverse triangle inequality for the $L^2$ norm bounds the corresponding squared difference in \eqref{eq:bl} by the conditional expectation of the squared change in $f$ produced by exchanging these two cards. After averaging over $S,a$, and $b$, we obtain
\[
\Ent_\nu(g^2)
\leq
\frac{c_2}{N}
\E_{\pi_N}\left[
\sum_{\substack{i\in I_1\\j\in I_2}}
\bigl(f(\tau_{i,j}^{(N)}\circ\sigma)-f(\sigma)\bigr)^2
\right].
\]

For $i<j$, the transposition $\tau_{i,j}^{(N)}$ can be written as \[
\tau_{i}^{(N)}\circ\tau_{i+1}^{(N)}\circ\cdots\tau_{j-1}^{(N)}\cdots\tau_{i+1}^{(N)}\circ\tau_i^{(N)}, \] which is the composition of the sequence of adjacent transpositions with edge indices $i,i+1,\ldots,j-1,j-2,\ldots,i$. This sequence has length $2(j-i)-1\leq 2N$. A fixed path edge belongs to at most $N^2$ intervals $[i,j]$, and its associated adjacent transposition appears at most twice in the sequence associated to any one of those intervals. Applying the Cauchy--Schwarz inequality along each sequence and using the invariance of $\pi_N$ under every intermediate permutation yields the inequality 
\[\sum_{1\leq i<j\leq N}\E_{\pi_N}\left[\bigl(f(\tau_{i,j}^{(N)}\circ\sigma)-f(\sigma)\bigr)^2\right]\leq c_3N^3\sum_{j=1}^{N-1}\E_{\pi_N}\left[
\bigl(f(\tau_j^{(N)}\circ\sigma)-f(\sigma)\bigr)^2 \right]. \]
It follows that
\[\Ent_\nu(g^2)\leq c_4N^2\mathcal D_N^{\mathrm{path}}(f).\]

Let $A_N$ be the smallest constant such that
\[
\Ent_{\pi_N}(f^2)\leq A_N\mathcal D_N^{\mathrm{path}}(f)
\]
for all $f\colon\SSS_N\to\mathbb R$. Conditional on $S$, the two subpaths with vertex sets $I_1$ and $I_2$ contain independent uniformly random permutations of their respective label sets. Applying the logarithmic Sobolev inequalities for these two smaller paths to the first term in the entropy decomposition above gives
\[
A_N\leq\max(A_k,A_{N-k})+c_4N^2.
\]
Since $A_1=0$ and the sizes of the subpaths decrease by a fixed factor under iteration, this recurrence implies that $A_N\leq c_1N^2$.
\end{proof}

We next transfer the estimate in \cref{lem:lsi} to the adjacent transposition shuffle on the cycle. Let
\[\mathcal D_n^{\mathrm{cyc}}(f)=\langle f,-\mathcal L_nf\rangle_{\pi_n}\]
be the Dirichlet form associated to the generator $\mathcal L_n$ in \eqref{eq:shuffle-generator}. Delete the cycle edge $\{n-1,0\}$, and relabel the remaining path as $\mathcal P_n$. The contribution of the remaining edges to $\mathcal D_n^{\mathrm{cyc}}(f)$ is exactly the Dirichlet form of the path chain, after this relabeling. Since the contribution from the deleted edge is nonnegative, \cref{lem:lsi} yields the inequality 
\begin{equation}\label{eq:cycle-lsi}
\Ent_{\pi_n}(f^2)\leq c_1n^2\mathcal D_n^{\mathrm{cyc}}(f).
\end{equation}

We will use \eqref{eq:cycle-lsi} after the relative entropy of the shuffle has already become small. For a probability measure $\alpha$ on $\SSS_n$, let $\alpha P_t$ denote the distribution obtained by starting the adjacent transposition shuffle on $\mathcal C_n$ with initial distribution $\alpha$ and running it for time $t$. Then
\begin{equation}\label{eq:entropy-decay}
H(\alpha P_t\mid\pi_n)
\leq
e^{-c_5\lambda_n t}H(\alpha\mid\pi_n)
\end{equation}
for some absolute constant $c_5>0$. To see this, fix $t>0$, and let
\[r_t=\frac{\mathrm{d}(\alpha P_t)}{\mathrm{d}\pi_n}. \]
Reversibility implies that
\[\frac{\mathrm{d}}{\mathrm{d}t}H(\alpha P_t\mid\pi_n)=-\frac12 \E_{\pi_n}\left[\sum_{x\in\ZZ/n\ZZ}\bigl(r_t(\tau_x\circ\sigma)-r_t(\sigma)\bigr)
\bigl(\log r_t(\tau_x\circ\sigma)-\log r_t(\sigma)\bigr)\right]. \]
For $u,v>0$, we have $(u-v)(\log u-\log v)\geq4(\sqrt u-\sqrt v)^2$. Consequently,
\[\frac{\mathrm{d}}{\mathrm{d}t}H(\alpha P_t\mid\pi_n)\leq -4\mathcal D_n^{\mathrm{cyc}}(\sqrt{r_t}). \]
Since $\E_{\pi_n}[r_t]=1$, we have
\[\Ent_{\pi_n}(r_t)=H(\alpha P_t\mid\pi_n). \] Hence, applying \eqref{eq:cycle-lsi} to $\sqrt{r_t}$ yields the inequality \[
\frac{\mathrm{d}}{\mathrm{d}t}H(\alpha P_t\mid\pi_n)\leq -\frac{c_6}{n^2}H(\alpha P_t\mid\pi_n). \]
Because $\lambda_n\asymp n^{-2}$, this implies \eqref{eq:entropy-decay}. The calculation is valid for every positive time; letting the starting time decrease to $0$ implies the result for an arbitrary initial distribution $\alpha$.

We will also use a consequence of \cref{lem:lsi} for processes in which some card labels are identified. Fix $s\geq1$, a finite set $\mathcal A$ of colors, and a tuple $(k_a)_{a\in\mathcal A}$ of nonnegative integers satisfying $\sum_{a\in\mathcal A}k_a=s$. Let $\Omega_{s,\mathbf k}$ be the set of assignments of colors to the vertices of $\mathcal P_s$ in which exactly $k_a$ vertices have color $a$ for every $a\in\mathcal A$. Consider the continuous-time Markov chain on $\Omega_{s,\mathbf k}$ in which, for each edge of $\mathcal P_s$, the colors at its endpoints are exchanged at rate $1$. Let $\nu_{s,\mathbf k}$ be the uniform probability measure on $\Omega_{s,\mathbf k}$, and let $\mathcal D_{s,\mathbf k}^{\mathrm{col}}$ be the Dirichlet form of this chain. Then
\begin{equation}\label{eq:colored-poincare}
\Var_{\nu_{s,\mathbf k}}(f)
\leq
c_7s^2\mathcal D_{s,\mathbf k}^{\mathrm{col}}(f)
\end{equation}
for every real-valued function $f$ on $\Omega_{s,\mathbf k}$, with the same type of absolute constant independent of the number of colors and of the values of the $k_a$. Indeed, linearizing the inequality in \cref{lem:lsi} around a constant function gives a $c_7s^2$ Poincar\'e inequality for the path chain on $\SSS_s$. Fix a partition of the $s$ card labels into color classes of sizes $(k_a)_{a\in\mathcal A}$. Mapping a labeled permutation to the induced coloring of the path sends the uniform measure on $\SSS_s$ to $\nu_{s,\mathbf k}$. If $f$ is a function of the coloring and $\widetilde f$ is its pullback to $\SSS_s$, then we have 
\[
\Var_{\pi_s}(\widetilde f)=\Var_{\nu_{s,\mathbf k}}(f)
\quad\text{and}\quad
\mathcal D_s^{\mathrm{path}}(\widetilde f)
=
\mathcal D_{s,\mathbf k}^{\mathrm{col}}(f).
\]
Thus, the Poincar\'e inequality for the labeled path chain gives \eqref{eq:colored-poincare}.

\subsection{The One-Card Walk}

We now study the motion of a single card in the adjacent transposition shuffle on $\mathcal C_n$. If we follow one specified card and ignore all other labels, its position is a continuous-time random walk on $\ZZ/n\ZZ$. From any vertex, it jumps to each of its two neighbors at rate $1$. Thus, its generator is
\[
\Delta f(x)=f(x-1)+f(x+1)-2f(x).
\]
For $i,x\in\ZZ/n\ZZ$ and $t\geq 0$, let
$p_t(i,x)$ be the probability that this walk is at $x$ at time $t$ when it starts at $i$.

For $0\leq j<n$, define
$\lambda_{n,j}=2\bigl(1-\cos(2\pi j/n)\bigr)$. These are the eigenvalues of $-\Delta$. In particular, the quantity $\lambda_n$ defined in \eqref{eq:cutoff-time} is $\lambda_{n,1}=\lambda_{n,n-1}$. Fourier inversion gives
\begin{equation}\label{eq:fourier}
p_t(0,x)
=
\frac1n
\sum_{j=0}^{n-1}
e^{-\lambda_{n,j}t}e^{2\pi\mathrm{i}jx/n},
\end{equation}
where $\mathrm{i}^2=-1$. The terms indexed by $j$ and $n-j$ are complex conjugates, so the right-hand side is real.

Let
\[
h(t)
=
\sum_{x\in\ZZ/n\ZZ}
p_t(0,x)\log\bigl(np_t(0,x)\bigr).
\]
Thus, $h(t)$ is the relative entropy of the distribution $p_t(0,\cdot)$ with respect to the uniform probability measure on $\ZZ/n\ZZ$.

\begin{lemma}\label[lemma]{lem:heat}
There are absolute constants $a_0,c_8>0$ such that, for every $n\geq3$ and every $t\geq a_0/\lambda_n$,
\begin{equation}\label{eq:heat-bounds}
\frac12\leq np_t(i,x)\leq\frac32
\qquad\text{for all }i,x\in\ZZ/n\ZZ,
\qquad
h(t)\leq c_8e^{-2\lambda_n t}.
\end{equation}
Moreover, we have
\begin{equation}\label{eq:B}
0\leq
B_n
\coloneq
\int_0^\infty
\left(
S_3(t)-\frac{S_2(t)}n
\right)\,\mathrm{d}t
\leq
c_8(1+\log n),
\end{equation}
where \[
S_2(t)=\sum_{x\in\ZZ/n\ZZ}p_t(0,x)^2
\quad\text{and}\quad
S_3(t)=\sum_{x\in\ZZ/n\ZZ}p_t(0,x)^3.
\]
\end{lemma}

\begin{proof}
For $0\leq j<n$, let $j_*=\min(j,n-j)$.
Since
\[\lambda_{n,j}=4\sin^2\left(\frac{\pi j_*}{n}\right)
\quad\text{and}\quad
\lambda_n=4\sin^2\left(\frac{\pi}{n}\right), \]
the elementary bounds
\[
\frac{2u}{\pi}\leq\sin u\leq u\quad(0\leq u\leq\pi/2)
\]
imply that 
\[\frac{\lambda_{n,j}}{\lambda_n}\geq
\frac{4j_*^2}{\pi^2}. \]
It follows from \eqref{eq:fourier} that
\[\left|np_t(0,x)-1\right|
\leq
\sum_{j=1}^{n-1}e^{-\lambda_{n,j}t}.
\]
If $t\geq a_0/\lambda_n$, the preceding eigenvalue bound shows that the right-hand side is at most
\[
2\sum_{r=1}^\infty e^{-4a_0r^2/\pi^2}.
\]
Choosing the absolute constant $a_0$ sufficiently large makes this quantity at most $1/2$. Translation invariance of the walk then gives the bounds 
\[
\frac12\leq np_t(i,x)\leq\frac32
\]
for all $i,x$.

For $t\geq a_0/\lambda_n$, the same eigenvalue estimate, now separating the two modes with eigenvalue $\lambda_n$ from the remaining modes, gives the estimate 
\[\max_x\left|np_t(0,x)-1\right|
\leq c_9e^{-\lambda_n t}.
\]
Applying Parseval's identity to \eqref{eq:fourier}, we deduce that 
\[nS_2(t)-1=\sum_{j=1}^{n-1}e^{-2\lambda_{n,j}t}\leq
c_{10}e^{-2\lambda_n t}. \]
Since $\log u\leq u-1$ for $u>0$, we have 
\[h(t)\leq\sum_xp_t(0,x)\bigl(np_t(0,x)-1\bigr)=nS_2(t)-1. \]
This proves \eqref{eq:heat-bounds}.

We next prove \eqref{eq:B}. Since $\sum_xp_t(0,x)=1$, it follows from the Cauchy--Schwarz inequality that 
\[S_3(t)=\sum_xp_t(0,x)^3\geq\left(\sum_xp_t(0,x)^2\right)^2=S_2(t)^2. \]
Also, $S_2(t)\geq1/n$, so
\[S_3(t)\geq S_2(t)^2\geq\frac{S_2(t)}n.\]
Hence, the integrand defining $B_n$ is nonnegative. We now bound its integral. The Fourier representation and the preceding eigenvalue estimates imply the standard bound
\[\max_xp_t(0,x)\leq\frac{c_{11}}{\sqrt{1+t}}+\frac{c_{11}}{n}\] for $t\geq 0$.
Indeed, the Fourier sum can be bounded by a Gaussian sum of the form
\[
\frac{c_{12}}{n}\sum_{r\geq0}e^{-c_{13}r^2t/n^2}.
\]
Because $\sum_xp_t(0,x)=1$, we have
\[S_3(t)\leq\bigl(\max_xp_t(0,x)\bigr)^2.\]
Consequently, for $0\leq t\leq n^2$,
\[
0\leq
S_3(t)-\frac{S_2(t)}n
\leq
\frac{c_{14}}{1+t}+\frac{c_{14}}{n^2}.
\]
Integrating over $[0,n^2]$ gives a bound of order $1+\log n$.

It remains to control the tail. For $t\geq n^2$, write
\[
p_t(0,x)=\frac{1+r_x(t)}{n}.
\]
Then
\[
\sum_xr_x(t)=0.
\]
The Fourier representation gives the inequality 
\[
\max_x|r_x(t)|\leq c_{15}e^{-\lambda_n t}.
\]
A direct expansion yields the identity 
\[
S_3(t)-\frac{S_2(t)}n
=
\frac1{n^3}
\sum_x
\left(2r_x(t)^2+r_x(t)^3\right).
\]
Since $r_x(t)\geq-1$, every summand on the right is nonnegative, and the preceding bound on $r_x(t)$ gives the bound 
\[
S_3(t)-\frac{S_2(t)}n
\leq
\frac{c_{16}}{n^2}e^{-2\lambda_n t}.
\]
Finally, $\lambda_n\asymp n^{-2}$, so the integral of this bound over $[n^2,\infty)$ is bounded by an absolute constant. Combining the two time ranges proves \eqref{eq:B}. 
\end{proof}
 
\section{Random Exposure and Conditional Prediction}\label{sec:exposure}

We use the entropy chain rule by revealing the positions of the cards one at a time. The order in which the card labels are revealed will itself be chosen uniformly at random. At each stage, we compare the conditional distribution of the next card position with a probability distribution obtained from the 1-card transition probabilities introduced in \cref{sec:prelim}.

\subsection{An Entropy Comparison}

Let $W\colon(\ZZ/n\ZZ)\times(\ZZ/n\ZZ)\to\mathbb R_{>0}$ satisfy
\[
\sum_{x\in\ZZ/n\ZZ}W(i,x)=n
\]
for every $i\in\ZZ/n\ZZ$. Let $\mu$ be a probability measure on $\SSS_n$, and let $\sigma$ be a random permutation with law $\mu$. Independently of $\sigma$, choose a uniformly random ordering $(J_1,\ldots,J_n)$ of the card labels.

For $m\in[n]$, define
$i_m=J_{n-m+1}$,
and let
$U_m=\{J_{n-m+1},\ldots,J_n\}$
be the set of labels that have not yet been exposed when $m$ labels remain. Let
\[
R_m=\{\sigma(i):i\in U_m\}
\]
be the corresponding set of unexposed positions. Thus, $|U_m|=|R_m|=m$, and $i_m\in U_m$ is the next label to be exposed. Let $\mathcal F_m$ be the sigma-algebra generated by the entire ordering $(J_1,\ldots,J_n)$ together with the positions $\sigma(J_1),\ldots,\sigma(J_{n-m})$ of the labels that have already been exposed. In particular, $i_m$ and $R_m$ are $\mathcal F_m$-measurable.

For $i\in\ZZ/n\ZZ$ and a nonempty set $R\subseteq\ZZ/n\ZZ$, define the probability measure $\widehat q_{i,R}$ on $R$ by
\[
\widehat q_{i,R}(x)=\frac{W(i,x)}{\sum_{y\in R}W(i,y)}.
\]
All expectations in the next lemma are taken with respect to the joint law of $\sigma$ and the independent random ordering $(J_1,\ldots,J_n)$.

\begin{lemma}\label[lemma]{lem:exposure}
Suppose there are constants $0<\ell\leq M<\infty$ such that
$\ell\leq W(i,x)\leq M$
for all $i,x\in\ZZ/n\ZZ$. Then there is a constant $K_{\ell,M}>0$, depending only on $\ell$ and $M$, such that
\begin{equation}\label{eq:exposure}
\begin{split}
H(\mu\mid\pi_n)\leq{}&\sum_{m=1}^n\E\left[H\bigl(\mu(\sigma(i_m)\in\cdot\mid\mathcal F_m)\mid\widehat q_{i_m,R_m}\bigr)\right]\\
&+\sum_{i\in\ZZ/n\ZZ}\E_\mu[\log W(i,\sigma(i))]+K_{\ell,M}(1+\log n).
\end{split}
\end{equation}
\end{lemma}

\begin{proof}
For $m\in[n]$, define the positive random variable
\[
Z_m=\frac1m\sum_{x\in R_m}W(i_m,x).
\]
Under the uniform measure $\pi_n$, conditional on $\mathcal F_m$, the position of card $i_m$ is uniformly distributed on $R_m$. Applying the entropy chain rule along the exposure ordering and then rewriting the uniform conditional measure on $R_m$ relative to $\widehat q_{i_m,R_m}$ yields the identity
\begin{align*}
H(\mu\mid\pi_n)={}&\sum_{m=1}^n\E\left[H\bigl(\mu(\sigma(i_m)\in\cdot\mid\mathcal F_m)\mid\widehat q_{i_m,R_m}\bigr)\right]\\
&+\sum_{i\in\ZZ/n\ZZ}\E_\mu[\log W(i,\sigma(i))]-\E\left[\sum_{m=1}^n\log Z_m\right].
\end{align*}
We are left to obtain a lower bound for the last expectation.

Fix a permutation $\sigma$, and condition on the event $i_m=i$. Under this conditioning, $R_m$ is uniformly distributed among the $m$-element subsets of $\ZZ/n\ZZ$ that contain $\sigma(i)$. Since the $i$-th row of $W$ has sum $n$, direct sampling without replacement yields the identity
\[
\E[Z_m-1\mid i_m=i,\sigma]
=
\frac{n-m}{m(n-1)}\bigl(W(i,\sigma(i))-1\bigr).
\]
The variance formula for sampling without replacement and the bounds $\ell\leq W\leq M$ yield
\[
\E[(Z_m-1)^2\mid i_m=i,\sigma]\leq\frac{K_{\ell,M}}m
\]
after increasing $K_{\ell,M}$ if necessary. Because $Z_m\in[\ell,M]$, Taylor's theorem applied to $\log z$ on the interval $[\ell,M]$ yields
\[
\left|\E[\log Z_m\mid i_m=i,\sigma]\right|\leq\frac{K_{\ell,M}}m.
\]
This estimate is uniform in $i$ and $\sigma$. Averaging over $i_m$ and $\sigma$ and then summing over $m$ yields
\[
\left|\E\left[\sum_{m=1}^n\log Z_m\right]\right|\leq K_{\ell,M}(1+\log n).
\]
Substituting this inequality into the preceding entropy identity proves \eqref{eq:exposure}.
\end{proof}

We will apply \cref{lem:exposure} with
$W(i,x)=np_t(i,x)$ and $\mu=\mu_t$. Under $\mu_t$, the position of card $i$ has distribution $p_t(i,\cdot)$. Therefore, translation invariance of the 1-card walk yields the identity
\begin{equation}\label{eq:1-card-log-term}
\sum_{i\in\ZZ/n\ZZ}\E_{\mu_t}[\log W(i,\sigma(i))]=nh(t).
\end{equation}
We next bound the conditional relative entropies in \eqref{eq:exposure}.

\subsection{Conditioning on Revealed Trajectories}\label{subsec:revealed}

Fix $m\in[n]$. Let $A$ be a uniformly random $m$-element subset of the card labels, chosen independently of all Poisson clocks used to construct the shuffle. We call the labels in $A$ \dfn{red} and the labels in $(\ZZ/n\ZZ)\setminus A$ \dfn{black}. Since the shuffle starts from the identity permutation, the red labels initially occupy the vertex set $A$.

Let $(\sigma_s)_{s\geq0}$ denote the adjacent transposition shuffle on the cycle $\mathcal C_n$ started from the identity. For $s\geq 0$, define
\[
R_s=\{\sigma_s(i):i\in A\}.
\]
Thus, $R_s$ is the set of vertices occupied by red cards at time $s$. The process $(R_s)_{s\geq0}$ is the $m$-particle simple exclusion process on $\mathcal C_n$. Because $A$ is uniformly distributed among the $m$-element subsets of $\ZZ/n\ZZ$, the initial state $R_0$ has the uniform stationary distribution of this exclusion process. Consequently, $R_s$ is uniformly distributed among the $m$-element subsets of $\ZZ/n\ZZ$ for every fixed $s\geq0$.

Fix $t\geq0$. Let $\mathcal G_t$ be the sigma-algebra generated by $A$ and by the trajectories
$(\sigma_s(j))_{0\leq s\leq t}$
of all black labels $j\notin A$. The path $(R_s)_{0\leq s\leq t}$ is $\mathcal G_t$-measurable. For $i\in A$ and $x\in\ZZ/n\ZZ$, define
\begin{equation}\label{eq:conditional-kernel}
K_t(i,x)=\Prob(\sigma_t(i)=x\mid\mathcal G_t).
\end{equation}
Because every red card occupies a vertex of $R_t$, we have $K_t(i,x)=0$ whenever $x\notin R_t$.

Retain the positions of all black cards, and replace every red label by a single blank red card (see \cref{fig:forget-red}). A transition changes this projection exactly when the exchanged edge has at least one endpoint occupied by a black card. The rate of each such transition, and hence their total rate, is independent of the individual red labels. Conditional on $\mathcal G_t$, the positions of the red cards on
$[0,t]$ have the following distribution. Start each red card $i\in A$ at vertex~$i$, and use independent Poisson clocks of rate~$1$, one for each edge of $\mathcal C_n$. When a clock rings, exchange the red cards at its endpoints if both endpoints belong to $R_s$, and otherwise make no change. At each jump of $(R_s)_{0\le s\le t}$, move the red card at
the unique vertex leaving $R_s$ to the unique vertex entering $R_s$.
The independence of the transition rates described above from the individual red labels ensures that conditioning on the trajectories of black cards, including the times of their jumps, introduces no further change to this distribution. This construction depends only on $(R_s)_{0\le s\le t}$. Since $A=R_0$, it follows that for $i\in A$,
we have \[K_t(i,x)=\mathbb P\bigl(\sigma_t(i)=x\mid (R_s)_{0\le s\le t}\bigr). \]

\begin{figure}[ht]
\begin{center}{\includegraphics[height=37.390mm]{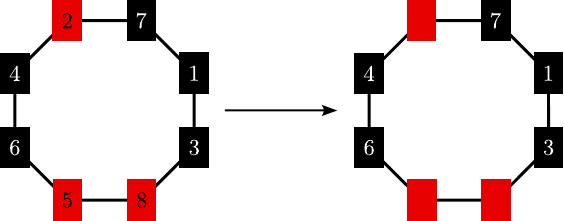}}
\end{center}
\caption{An illustration with $n=8$, $m=3$, and $A=\{2,5,8\}$. The projection forgets the labels of the red cards. }   
\label{fig:forget-red}
\end{figure}

For $i\in A$, the vector $((n/m)p_t(i,x))_{x\in R_t}$ is not necessarily a probability vector because its coordinates need not sum to $1$. We nevertheless compare $K_t(i,\cdot)$ with these unnormalized reference masses. Define
\begin{equation}\label{eq:error-definition}
\mathcal E_{n,m}(t)
=
\E\left[
\frac1m\sum_{i\in A}\sum_{x\in R_t}
\left(K_t(i,x)-\frac nm p_t(i,x)\right)^2
\right],
\end{equation}
where the expectation is taken over the random set $A$ and the Poisson clocks of the shuffle.

\begin{proposition}\label[proposition]{prop:entropy-error}
There is an absolute constant $K_0>0$ such that, for every $t\geq a_0/\lambda_n$, where $a_0$ is the constant in \cref{lem:heat},
\begin{equation}\label{eq:entropy-error}
H(\mu_t\mid\pi_n)
\leq
nh(t)+K_0(1+\log n)+2\sum_{m=1}^n m\mathcal E_{n,m}(t).
\end{equation}
\end{proposition}

\begin{proof}
We first record an elementary inequality. Let $u=(u_x)_{x\in\Lambda}$ be a probability vector on a finite set $\Lambda$, and let $v=(v_x)_{x\in\Lambda}$ have strictly positive coordinates. The vector $v$ is not assumed to have total mass $1$. Then
\begin{equation}\label{eq:unnormalized-kl}
H\left(u\;\middle|\;\frac{v}{\sum_{x\in\Lambda}v_x}\right)
\leq
\sum_{x\in\Lambda}\frac{(u_x-v_x)^2}{v_x}.
\end{equation}
Indeed, if $r=\sum_xv_x$, then the inequality $\log z\leq z-1$ yields
\[
H\left(u\;\middle|\;\frac{v}{r}\right)
\leq
\sum_x\frac{u_x^2}{v_x}-1+\log r.
\]
Since $\log r\leq r-1$, the right-hand side is at most
\[
\sum_x\frac{(u_x-v_x)^2}{v_x},
\]
which proves \eqref{eq:unnormalized-kl}.

Apply \cref{lem:exposure} to a random permutation with law $\mu_t$, taking $W(i,x)=np_t(i,x)$. At the stage with $m$ unexposed labels, let $A=U_m$ be the set of those labels. Let $\mathcal H_m$ be the sigma-algebra generated by $\mathcal F_m$ together with the complete trajectories up to time $t$ of all labels outside $A$. Then $\mathcal F_m\subseteq\mathcal H_m$, and the predictor $\widehat q_{i_m,R_m}$ is $\mathcal F_m$-measurable. Conditional on $\mathcal H_m$, the distribution of the position of card $i_m$ is $K_t(i_m,\cdot)$; the additional information in the exposure ordering concerns only the order in which the red labels will be exposed and is independent of the shuffle. Conditional Jensen's inequality for relative entropy therefore bounds the $m$-th term in \eqref{eq:exposure} by
\[
\E\left[H\bigl(K_t(i_m,\cdot)\mid\widehat q_{i_m,R_t}\bigr)\right].
\]
Conditional on the unordered set $A$ and on the black-card trajectories, the label $i_m$ is uniform on $A$. Averaging over this choice yields the identity
\begin{equation}\label{eq:conditional-kl-bound}
\E\left[H\bigl(K_t(i_m,\cdot)\mid\widehat q_{i_m,R_t}\bigr)\right]
=
\E\left[
\frac1m\sum_{i\in A}
H\bigl(K_t(i,\cdot)\mid\widehat q_{i,R_t}\bigr)
\right].
\end{equation}

For fixed $A$, $i\in A$, and $R_t$, apply \eqref{eq:unnormalized-kl} on the finite set $R_t$ with
\[
u_x=K_t(i,x)
\qquad\text{and}\qquad
v_x=\frac nm p_t(i,x).
\]
By \cref{lem:heat}, the assumption $t\geq a_0/\lambda_n$ implies that $v_x\geq 1/(2m)$
for every $x\in R_t$. Hence, the quantity in \eqref{eq:conditional-kl-bound} is at most
$2m\mathcal E_{n,m}(t)$.
Finally, \cref{lem:heat} also yields $1/2\leq W(i,x)\leq3/2$, while \eqref{eq:1-card-log-term} identifies the 1-card term in \eqref{eq:exposure} with $nh(t)$. Applying \cref{lem:exposure} with $\ell=1/2$ and $M=3/2$ proves \eqref{eq:entropy-error} after renaming the resulting absolute constant as $K_0$.
\end{proof}

\section{The Two-Copy Process}\label{sec:replica}

We next represent the prediction error $\mathcal E_{n,m}(t)$ using a finite reversible Markov chain. Each state records the set of vertices occupied by red cards and two additional vertices belonging to that set. The two vertex coordinates initially coincide at a vertex chosen uniformly from the set of vertices occupied by red cards at time~$0$. For every $t\ge0$, conditional on this initial vertex and $(R_s)_{0\le s\le t}$, the two coordinate processes on $[0,t]$ are independent, and each has the same conditional distribution as the position of the red card that started at that vertex. 

\subsection{Definition of the Two-Copy Process}

Fix $m\in[n]$, and define
\[
\Omega_{n,m}
=
\{(R,x,y):R\subseteq\ZZ/n\ZZ,\ |R|=m,\ x,y\in R\}.
\]
The coordinates $x$ and $y$ are allowed to be equal. For a cycle edge $e=\{u,v\}$, let $\tau_e$ denote the transposition of the vertices $u$ and $v$. We define a continuous-time Markov chain on $\Omega_{n,m}$ by specifying the transitions associated with each edge $e$.

If $|e\cap R|=1$, then at rate $1$ the state $(R,x,y)$ is replaced by
\[
(\tau_e(R),\tau_e(x),\tau_e(y)),
\]
where $\tau_e(R)=\{\tau_e(z):z\in R\}$. Thus, the unique red endpoint of $e$ moves to the other endpoint, and each auxiliary coordinate located at that red endpoint is transported with it. If $e\subseteq R$, then $R$ does not change. If $x\in e$, the state $(R,x,y)$ is replaced at rate $1$ by $(R,\tau_e(x),y)$. Independently, if $y\in e$, the state $(R,x,y)$ is replaced at rate $1$ by $(R,x,\tau_e(y))$. If $e\cap R=\varnothing$, the edge produces no transition. Let $L_{n,m}$ denote the generator of this chain. 

Every nontrivial transition has a reverse transition with the same rate. Hence, the chain is reversible with respect to the uniform probability measure $\varpi_{n,m}$ on $\Omega_{n,m}$. Since
\[
|\Omega_{n,m}|=\binom nm m^2,
\]
every state has $\varpi_{n,m}$-mass $1/[\binom nm m^2]$.

The chain is irreducible. To verify this, consider for a cycle edge $e$ the map
\[
(R,x,y)\longmapsto(\tau_e(R),\tau_e(x),\tau_e(y)).
\]
If $|e\cap R|=1$, this map is one of the transitions defined above. If $e\subseteq R$, its nontrivial action on $x$ and $y$ can be implemented by at most two successive single-coordinate moves. Because adjacent transpositions on the cycle generate the full symmetric group on the vertex set, compositions of these maps act transitively on the states satisfying $x=y$ and also transitively on the states satisfying $x\neq y$. When $m\geq2$, a state with $x=y$ can be transformed so that a second red vertex is adjacent to $x$, after which one auxiliary coordinate can move across that red-red edge while the other remains at $x$. Thus, the sets $\{x=y\}$ and $\{x\neq y\}$ communicate. When $m=1$, every state satisfies $x=y$, and the first transitivity statement proves irreducibility.

We now specify the initial distribution that will be used below. Choose $R_0$ uniformly from the $m$-element subsets of $\ZZ/n\ZZ$, choose $I$ uniformly from $R_0$, and set $X_0=Y_0=I$.
Let $(R_t,X_t,Y_t)_{t\geq0}$ denote the two-copy process with this initial distribution, and define
\begin{equation}\label{eq:coincidence-probability}
C_{n,m}(t)=\Prob(X_t=Y_t).
\end{equation}

Conditional on $(R_s)_{0\le s\le t}$ and on $I=i$, the coordinates $X_t$ and $Y_t$ are independent, and each has conditional distribution $K_t(i,\cdot)$. Indeed, whenever $R_s$ changes, each coordinate at the vertex leaving $R_s$ moves to the vertex entering $R_s$. Between these changes, each coordinate jumps from its current vertex to each adjacent vertex in $R_s$ at rate~$1$, using independent Poisson clocks for the two coordinates. These transitions agree with the conditional motion of card~$i$ described in \cref{subsec:revealed}. 

\begin{lemma}\label[lemma]{lem:error-identity}
For every $t\geq 0$, we have 
\begin{equation}\label{eq:error-identity}
\mathcal E_{n,m}(t)
=
C_{n,m}(t)-\frac nm S_2(t)
+\left(\frac nm\right)^2\frac{n-m}{n-1}
\left(S_3(t)-\frac{S_2(t)}n\right).
\end{equation}
\end{lemma}

\begin{proof}
Expand the square in \eqref{eq:error-definition}. We evaluate the three resulting terms separately.

For the term containing $K_t(i,x)^2$, conditional independence of $X_t$ and $Y_t$ yields
\[
\E\left[\frac1m\sum_{i\in A}\sum_x K_t(i,x)^2\right]
=
C_{n,m}(t).
\]

For the cross term, the random set $A$ is independent of the Poisson clocks. Conditional on $i\in A$, averaging $K_t(i,x)$ over the black trajectories therefore recovers the unconditional 1-card transition probability $p_t(i,x)$. Translation invariance then yields
\[
\E\left[\frac1m\sum_{i\in A}\sum_xK_t(i,x)p_t(i,x)\right]
=
S_2(t).
\]
Thus, the cross term equals $-2(n/m)S_2(t)$.

It remains to evaluate the squared reference term. Set $\alpha=(m-1)/(n-1)$. Fix $i,x\in\ZZ/n\ZZ$, and condition on $i\in A$. If card $i$ occupies $x$ at time $t$, then $x\in R_t$. If card $i$ does not occupy $x$, then the label occupying $x$ is one of the other $n-1$ labels and belongs to $A$ with conditional probability $\alpha$. Therefore,
\begin{equation}\label{eq:red-site-probability}
\Prob(x\in R_t\mid i\in A)
=
\alpha+(1-\alpha)p_t(i,x).
\end{equation}
Using \eqref{eq:red-site-probability} and translation invariance, the squared reference term equals
\[
\left(\frac nm\right)^2
\left[\alpha S_2(t)+(1-\alpha)S_3(t)\right].
\]
Combining the three terms yields \eqref{eq:error-identity}.
\end{proof}

\subsection{The Inverse-Generator Norm}

Let
\[
\langle f,g\rangle_{\varpi_{n,m}}
=
\E_{\varpi_{n,m}}[fg]
\]
for functions on $\Omega_{n,m}$. Since $L_{n,m}$ is the generator of a finite irreducible reversible Markov chain, the operator $-L_{n,m}$ is positive definite on the subspace of functions with $\varpi_{n,m}$-mean $0$. For such a function $g$, define $(-L_{n,m})^{-1}g$ to be the unique mean-0 function $u$ satisfying
$-L_{n,m}u=g$.
Define
\begin{equation}\label{eq:hminus1-definition}
\|g\|_{-1}^2
=
\langle g,(-L_{n,m})^{-1}g\rangle_{\varpi_{n,m}}
\end{equation}
and
\[
\mathcal D_{n,m}(f)
=
\langle f,-L_{n,m}f\rangle_{\varpi_{n,m}}.
\]
Reversibility implies $\langle f,-L_{n,m}g\rangle_{\varpi_{n,m}}=\langle -L_{n,m}f,g\rangle_{\varpi_{n,m}}$. Hence, $-L_{n,m}$ has an orthonormal eigenbasis on the mean-0 subspace. Expanding in this eigenbasis and completing the square yields the variational identity
\begin{equation}\label{eq:variational}
\|g\|_{-1}^2
=
\sup_f\left\{2\langle g,f\rangle_{\varpi_{n,m}}-\mathcal D_{n,m}(f)\right\},
\end{equation}
where the supremum may be taken over all real-valued functions on $\Omega_{n,m}$ because adding a constant to $f$ changes neither term. The same spectral decomposition yields the identity
\begin{equation}\label{eq:correlation-representation}
\|g\|_{-1}^2
=
\int_0^\infty
\langle g,e^{tL_{n,m}}g\rangle_{\varpi_{n,m}}\,\mathrm{d}t.
\end{equation}
If the two-copy process starts from its stationary law $\varpi_{n,m}$, then the integrand in \eqref{eq:correlation-representation} is $\E[g(R_0,X_0,Y_0)g(R_t,X_t,Y_t)]$. Thus, \eqref{eq:correlation-representation} expresses $\|g\|_{-1}^2$ as the time integral of this expectation. We use only the finite-state identities \eqref{eq:variational} and \eqref{eq:correlation-representation}; for the broader role of inverse-generator norms in the study of additive functionals of reversible processes, see the work of Kipnis and Varadhan \cite{KV}.

For $(R,X,Y)\in\Omega_{n,m}$, define
\begin{equation}\label{eq:source}
b(R,X,Y)
=
\1_{\{X=Y\}}
\left(
\1_{\{X-1\in R\}}+\1_{\{X+1\in R\}}-\frac{2(m-1)}{n-1}
\right)
\end{equation}
and set
\begin{equation}\label{eq:Rnm-definition}
\mathscr R_{n,m}=\|b\|_{-1}^2.
\end{equation}
The function $b$ has $\varpi_{n,m}$-mean $0$. Indeed, conditional on $X=Y$, the set $R\setminus\{X\}$ is uniformly distributed among the $(m-1)$-element subsets of the remaining $n-1$ vertices, so the conditional expected number of red neighbors of $X$ is $2(m-1)/(n-1)$. The function $b$ is identically 0 when $m=1$ or $m=n$.

\begin{proposition}\label[proposition]{prop:integrated}
For every $1\leq m\leq n$, we have
\begin{equation}\label{eq:integrated-error}
\int_0^\infty\mathcal E_{n,m}(t)\,\mathrm{d}t
=
-\frac{(n+1)(n-m)}{24m^2}
+\frac{(n-1)^2}{4m}\mathscr R_{n,m}
+\left(\frac nm\right)^2\frac{n-m}{n-1}B_n.
\end{equation}
\end{proposition}

\begin{proof}
Let
$D=\1_{\{X=Y\}}$ and $v=D-1/m$.
Under $\varpi_{n,m}$, the event $\{D=1\}$ has probability $1/m$. The initial distribution used in \eqref{eq:coincidence-probability} is precisely $\varpi_{n,m}$ conditioned on this event. Consequently,
\begin{equation}\label{eq:coincidence-correlation}
C_{n,m}(t)-\frac1m
=
m\langle v,e^{tL_{n,m}}v\rangle_{\varpi_{n,m}}.
\end{equation}
Integrating \eqref{eq:coincidence-correlation} and using \eqref{eq:correlation-representation} yields the identity
\begin{equation}\label{eq:coincidence-integral-norm}
\int_0^\infty\left(C_{n,m}(t)-\frac1m\right)\mathrm{d}t
=
m\|v\|_{-1}^2.
\end{equation}

We now separate from $\|v\|_{-1}^2$ the contribution depending only on the separation of the two auxiliary coordinates. For a state $(R,X,Y)$, let $d\in\{0,1,\ldots,n-1\}$ be the representative of $X-Y$ modulo $n$, and let $w=\1_{\{X-1\in R\}}+\1_{\{X+1\in R\}}$. For a function $\phi\colon\ZZ/n\ZZ\to\mathbb R$, write
\[
\Delta\phi(r)=\phi(r-1)+\phi(r+1)-2\phi(r),
\]
with the arguments interpreted modulo $n$. We also write $\phi(d)$ for the function $(R,X,Y)\mapsto\phi(X-Y)$ on $\Omega_{n,m}$. The transition rules of the two-copy process imply the identity
\begin{equation}\label{eq:separation-generator}
L_{n,m}\phi(d)
=
2\Delta\phi(d)+(w-2)D\,\Delta\phi(0).
\end{equation}
To verify \eqref{eq:separation-generator}, first suppose $X\neq Y$. Each coordinate then changes the separation by one step at rate $1$ in each direction, regardless of whether the neighboring vertex is red or black, so the contribution is $2\Delta\phi(d)$. Now suppose $X=Y$. For each black neighbor, the corresponding edge transports both coordinates together and leaves $d$ unchanged. The term $2\Delta\phi(0)$ treats that edge as if the two coordinates moved separately, so one copy of $\Delta\phi(0)$ must be subtracted for each black neighbor. Since there are $2-w$ black neighbors, the correction is $(w-2)D\Delta\phi(0)$.

Under $\varpi_{n,m}$, we have 
\[
\Prob(d=0)=\frac1m\quad\text{and}\quad
\Prob(d=r)=\frac{m-1}{m(n-1)}.
\] for $r\neq 0$. 
For $0\leq r\leq n-1$, define 
\begin{equation}\label{eq:corrector}
\phi(r)
=
\frac{n(m-1)(n+1)}{24m^2}-\frac{r(n-r)}{4m}.
\end{equation}
A direct summation using the preceding distribution of $d$ shows that $\E_{\varpi_{n,m}}[\phi(d)]=0$. Its cyclic second difference satisfies
\[
\Delta\phi(r)=\frac1{2m}
\quad\text{for }r\neq 0\quad\text{and}\quad
\Delta\phi(0)=-\frac{n-1}{2m}.
\]
Substituting these values into \eqref{eq:separation-generator} yields the identity
\begin{equation}\label{eq:corrector-identity}
-L_{n,m}\phi
=
v+\frac{n-1}{2m}b.
\end{equation}

The function $b$ is orthogonal in $L^2(\varpi_{n,m})$ to every function of $d$. Indeed, $b$ vanishes unless $d=0$, and its conditional mean given $d=0$ is $0$. In particular,
$\langle b,\phi\rangle_{\varpi_{n,m}}=0$.
Write $a=(n-1)/(2m)$. We deduce from  \eqref{eq:corrector-identity} that $v=(-L_{n,m})\phi-ab$. Using self-adjointness of $-L_{n,m}$ and the preceding orthogonality, we find that 
\begin{align*}
\|v\|_{-1}^2
&=\langle v,(-L_{n,m})^{-1}v\rangle_{\varpi_{n,m}}\\
&=\langle v,\phi\rangle_{\varpi_{n,m}}+a^2\langle b,(-L_{n,m})^{-1}b\rangle_{\varpi_{n,m}}\\
&=\langle\phi,v\rangle_{\varpi_{n,m}}+\frac{(n-1)^2}{4m^2}\mathscr R_{n,m}.
\end{align*}
Since $v=D-1/m$ and $\phi$ has mean $0$,
we have \[
\langle\phi,v\rangle_{\varpi_{n,m}}
=
\E_{\varpi_{n,m}}[\phi D]
=
\frac{\phi(0)}m.
\]
Combining this identity with \eqref{eq:coincidence-integral-norm} and the value of $\phi(0)$ from \eqref{eq:corrector} yields
\begin{equation}\label{eq:return-integral}
\int_0^\infty\left(C_{n,m}(t)-\frac1m\right)\mathrm{d}t
=
\frac{n(m-1)(n+1)}{24m^2}
+\frac{(n-1)^2}{4m}\mathscr R_{n,m}.
\end{equation}

We also need the corresponding integral for the 1-card walk. Reversibility and the semigroup property imply that 
\[
S_2(t)=\sum_xp_t(0,x)^2=p_{2t}(0,0).
\]
The function $u$ defined by 
\[
u(r)=\frac{n^2-1}{24n}-\frac{r(n-r)}{4n}
\]
has mean $0$ under the uniform probability measure on $\ZZ/n\ZZ$ and satisfies
\[
-2\Delta u=\1_{\{0\}}-\frac1n.
\]
Since $u$ has mean $0$, spectral decomposition of $-2\Delta$ yields
\[
u(0)=\int_0^\infty \left(e^{2t\Delta}\left(\1_{\{0\}}-\frac1n\right)\right)(0)\,\mathrm{d}t.
\]
Because the expression inside the integral is $p_{2t}(0,0)-1/n=S_2(t)-1/n$, this identity yields
\begin{equation}\label{eq:rw-integral}
\int_0^\infty\left(S_2(t)-\frac1n\right)\mathrm{d}t
=
\frac{n^2-1}{24n}.
\end{equation}

Finally, subtract the stationary terms in \eqref{eq:error-identity}, and integrate over $t\in[0,\infty)$. The integrals converge because both finite-state chains are irreducible and reversible, while the term involving $S_3-S_2/n$ is integrable by \cref{lem:heat}. Substituting \eqref{eq:return-integral}, \eqref{eq:rw-integral}, and the definition of $B_n$ yields \eqref{eq:integrated-error}; the first two deterministic terms combine according to the identity
\[
\frac{n(m-1)(n+1)}{24m^2}
-\frac nm\frac{n^2-1}{24n}
=
-\frac{(n+1)(n-m)}{24m^2}. \qedhere
\]
\end{proof}

\section{A Uniform Local Poincar\'e Inequality}\label{sec:local-gap}

The remaining task in the upper-bound argument is to estimate $\mathscr R_{n,m}$. We begin with a Poincar\'e inequality for a path version of the two-copy process. We define this path process explicitly because it will be used conditionally inside spatial blocks in \cref{sec:blocks}.

Fix integers $s\geq1$ and $1\leq k\leq s$. Let
\[
\Omega^{\mathrm p}_{s,k}
=
\{(R,x,y):R\subseteq[s],\ |R|=k,\ x,y\in R\}.
\]
For a path edge $e=\{j,j+1\}$, let $\tau_e$ denote the transposition of $j$ and $j+1$. Define a continuous-time Markov chain on $\Omega^{\mathrm p}_{s,k}$ as follows. If $|e\cap R|=1$, then at rate $1$ replace $(R,x,y)$ by
\[
(\tau_e(R),\tau_e(x),\tau_e(y)).
\]
If $e\subseteq R$, then $R$ remains fixed; at rate $1$, the $x$-coordinate crosses $e$ when $x\in e$, and independently, at rate $1$, the $y$-coordinate crosses $e$ when $y\in e$. If $e\cap R=\varnothing$, the edge produces no transition. Let $L^{\mathrm p}_{s,k}$ be the generator of this chain. Every transition has a reverse transition with the same rate, so the uniform probability measure $\nu_{s,k}$ on $\Omega^{\mathrm p}_{s,k}$ is reversible. The same argument used for the cycle two-copy process shows that the path process is irreducible. Let
\[
\mathcal D^{\mathrm p}_{s,k}(f)
=
\langle f,-L^{\mathrm p}_{s,k}f\rangle_{\nu_{s,k}}
\]
be its Dirichlet form.

\begin{lemma}\label[lemma]{lem:local-gap}
There is an absolute constant $K_1>0$ such that, for every $s\geq1$, every $1\leq k\leq s$, and every real-valued function $f$ on $\Omega^{\mathrm p}_{s,k}$,
\begin{equation}\label{eq:local-gap}
\Var_{\nu_{s,k}}(f)
\leq
K_1s^2\mathcal D^{\mathrm p}_{s,k}(f).
\end{equation}
\end{lemma}

\begin{proof}
The assertion is immediate when $s=1$. Assume $s\geq2$.

For a path edge $e$, define the \dfn{shared transposition} of a state $(R,x,y)$ by
\[
(R,x,y)\longmapsto(\tau_e(R),\tau_e(x),\tau_e(y)).
\]
Consider the auxiliary continuous-time Markov chain that applies every nontrivial shared transposition at rate $1$. Let $\mathcal D_{\mathrm{sh}}$ denote its Dirichlet form with respect to the uniform measure $\nu_{s,k}$. If $|e\cap R|=1$, the shared transposition is an actual transition of the path two-copy process. If $e\subseteq R$, the shared transposition can be implemented by at most two successive single-coordinate transitions of the path two-copy process. Applying $(a+b)^2\leq2a^2+2b^2$ to this two-step implementation and then changing variables under the measure-preserving intermediate move yields
\begin{equation}\label{eq:shared-comparison}
\mathcal D_{\mathrm{sh}}(f)
\leq
2\mathcal D^{\mathrm p}_{s,k}(f).
\end{equation}

Let $D=\1_{\{X=Y\}}$. The shared-transposition chain preserves the value of $D$. Conditional on $D=1$, a state can be identified with a coloring of the path having one distinguished red vertex, $k-1$ ordinary red vertices, and $s-k$ black vertices. Under this identification, the shared-transposition chain is exactly the colored path process defined before \eqref{eq:colored-poincare}. Conditional on $D=0$, it is the colored path process with two distinct distinguished red vertices, $k-2$ ordinary red vertices, and $s-k$ black vertices. Therefore, \eqref{eq:colored-poincare} and \eqref{eq:shared-comparison} yield
\begin{equation}\label{eq:within-sector-variance}
\E_{\nu_{s,k}}[\Var(f\mid D)]
\leq
K_2s^2\mathcal D^{\mathrm p}_{s,k}(f)
\end{equation}
for an absolute constant $K_2>0$.

If $k=1$, then $D=1$ almost surely, so \eqref{eq:within-sector-variance} proves the lemma. Assume henceforth that $k\geq2$. By the variance decomposition,
\[
\Var_{\nu_{s,k}}(f)
=
\E_{\nu_{s,k}}[\Var(f\mid D)]
+
\Var_{\nu_{s,k}}(\E[f\mid D]).
\]
It remains to estimate the second term.

Set $N_{s,k}=\binom sk$.
Choose $R$ uniformly among the $k$-element subsets of $[s]$, choose $x$ uniformly from $R$, and then choose $y$ uniformly from $R\setminus\{x\}$. The state $(R,x,x)$ has the conditional law $\nu_{s,k}(\,\cdot\mid D=1)$, while $(R,x,y)$ has the conditional law $\nu_{s,k}(\,\cdot\mid D=0)$. Since $\nu_{s,k}(D=1)=1/k$, Jensen's inequality yields
\begin{equation}\label{eq:sector-coupling}
\Var_{\nu_{s,k}}(\E[f\mid D])
\leq
\frac1{N_{s,k}k^3}
\sum_{\substack{R\subseteq[s]\\|R|=k}}
\sum_{\substack{x,y\in R\\x\neq y}}
\bigl(f(R,x,x)-f(R,x,y)\bigr)^2.
\end{equation}

For each triple $(R,x,y)$ with $x\neq y$, we now specify a path of transitions in the two-copy chain from $(R,x,x)$ to $(R,x,y)$. Suppose first that $x<y$. Apply shared transpositions successively on the edges with left endpoints $y-1,y-2,\ldots,x+1$.
These moves transport the red occupation initially at $y$ to $x+1$ and do not move either auxiliary coordinate, because both coordinates remain at $x$. The vertices $x$ and $x+1$ are then both red. Move only the second auxiliary coordinate from $x$ to $x+1$. Finally, undo the preceding shared transpositions in reverse order. During this reversal, the second auxiliary coordinate is transported from $x+1$ to $y$. The terminal state is $(R,x,y)$. If $y<x$, use the reflected construction.

Every nonidentity step in this prescribed path is an allowed transition of the path two-copy process. Before the single-coordinate move, neither auxiliary coordinate lies on an edge used by a shared transposition. After that move, only the second auxiliary coordinate can lie on such an edge, so each nontrivial shared transposition during the reversal is itself an allowed red-black transport or a single-coordinate move across a red-red edge. The prescribed path has at most $2s$ transitions.

We next bound the congestion of this family of prescribed paths. Fix an oriented transition $z\to z'$ of the path two-copy chain. Along every prescribed path, the first auxiliary coordinate remains equal to its initial value $x$. Hence, $z\to z'$ determines $x$, and there are at most $s$ possible values of the target $y$. For fixed $x$ and $y$, each spatial edge occurs at most twice in the prescribed path. At any specified occurrence, the initial red set $R$ is uniquely determined by the current red set and the preceding prescribed moves. Therefore, counting occurrences, at most $2s$ prescribed paths use the fixed oriented transition.

Applying the Cauchy--Schwarz inequality along each prescribed path in \eqref{eq:sector-coupling}, and then summing with the preceding congestion bound, yields
\[
\Var_{\nu_{s,k}}(\E[f\mid D])
\leq
\frac{4s^2}{N_{s,k}k^3}
\sum_{z,z':\,q(z,z')>0}
\bigl(f(z')-f(z)\bigr)^2,
\]
where $q(z,z')$ denotes the transition rate of the path two-copy process. Every nonzero transition rate equals $1$, and every state has $\nu_{s,k}$-mass $1/(N_{s,k}k^2)$. Consequently, the preceding inequality implies
\begin{equation}\label{eq:between-sector-variance}
\Var_{\nu_{s,k}}(\E[f\mid D])
\leq
\frac{8s^2}{k}\mathcal D^{\mathrm p}_{s,k}(f)
\leq
8s^2\mathcal D^{\mathrm p}_{s,k}(f).
\end{equation}
Combining \eqref{eq:within-sector-variance} and \eqref{eq:between-sector-variance} proves \eqref{eq:local-gap} after increasing the absolute constant $K_1$.
\end{proof}

\section{A Multiscale Bound for the Inverse Generator}\label{sec:blocks}

We now estimate the quantity $\mathscr R_{n,m}$ from \eqref{eq:Rnm-definition}. Recall that the function $b$ in \eqref{eq:source} is supported on states for which the two auxiliary coordinates coincide. On that event, it measures the difference between the number of red neighbors of the common coordinate and its conditional mean. We decompose this local fluctuation across a hierarchy of spatial intervals and use \cref{lem:local-gap} inside each interval.

\begin{proposition}\label[proposition]{prop:resolvent}
There is an absolute constant $K_3>0$ such that, for every $1\leq m\leq n$,
\begin{equation}\label{eq:resolvent}
\mathscr R_{n,m}
\leq
K_3\frac{1-m/n}{m}(1+\log n)^2.
\end{equation}
\end{proposition}

We first state a sampling estimate. For integers $n\geq1$, $0\leq m\leq n$, and $0\leq s\leq n$, we say that a random variable $K$ has the \dfn{hypergeometric distribution with parameters $(n,m,s)$} if, for a fixed $s$-element subset $B$ of an $n$-element set and a uniformly random $m$-element subset $R$, we have $K=|R\cap B|$.
Equivalently,
\[
\Prob(K=r)
=
\frac{\binom sr\binom{n-s}{m-r}}{\binom nm}
\]
for every integer $r$ for which the binomial coefficients are nonzero.

\begin{lemma}\label[lemma]{lem:hypergeometric}
Suppose $n\geq3$, $1\leq m\leq n$, and $2\leq s\leq n$. Let $K$ have the hypergeometric distribution with parameters $(n,m,s)$, and let $\rho=m/n$ and $\alpha=(m-1)/(n-1)$.
Then
\begin{equation}\label{eq:hypergeometric}
\E\left[
\left(\frac{K-1}{s-1}-\alpha\right)^2
\1_{\{K\geq1\}}
\right]
\leq
\frac{8\rho(1-\rho)}s.
\end{equation}
\end{lemma}

\begin{proof}
Let $\mu=s\rho$ and $a_*=1+(s-1)\alpha$.
The standard first and second moments of a hypergeometric random variable yield
\begin{equation}\label{eq:hypergeometric-moments}
\Var(K)
=
s\rho(1-\rho)\frac{n-s}{n-1}
\end{equation}
and
\begin{equation}\label{eq:hypergeometric-bias}
|\E[K]-a_*|
=
(1-\rho)\frac{n-s}{n-1}.
\end{equation}

Suppose first that $\mu\geq1$. Combining \eqref{eq:hypergeometric-moments} and \eqref{eq:hypergeometric-bias}, and using $s\rho\geq1$, yields
\[
\E[(K-a_*)^2]
\leq
2s\rho(1-\rho).
\]
Since
\[
\frac{s}{(s-1)^2}\leq\frac4s
\] for $s\geq 2$,
dividing by $(s-1)^2$ yields \eqref{eq:hypergeometric} in this case.

Now suppose that $\mu\leq1$. Define the size-biased version $K^{\mathrm{sb}}$ of $K$ by
\[
\Prob(K^{\mathrm{sb}}=j)
=
\frac{j\Prob(K=j)}{\E[K]}.
\]
The variable $K^{\mathrm{sb}}$ has the same distribution as
\[
1+\operatorname{Hypergeometric}(n-1,m-1,s-1),
\]
and its mean is $a_*$. Since $\1_{\{K\geq1\}}\leq K$, we obtain that 
\[\E[(K-a_*)^2\1_{\{K\geq1\}}]
\leq\E[K(K-a_*)^2]
=\mu\Var(K^{\mathrm{sb}}).\]
Applying the hypergeometric variance formula to $K^{\mathrm{sb}}$, and using
\[
\alpha\leq\rho,
\qquad
1-\alpha=\frac{n(1-\rho)}{n-1},
\qquad
\frac{n-s}{n-2}\leq 1
\]
yields
\[
\mu\Var(K^{\mathrm{sb}})
\leq
\frac32\rho^2s(s-1)(1-\rho).
\]
After division by $(s-1)^2$, the resulting upper bound is at most
\[
3\rho^2(1-\rho)
\leq
\frac{3\rho(1-\rho)}s,
\]
where the last inequality uses the fact that $s\rho=\mu\leq1$. This proves \eqref{eq:hypergeometric}. The endpoint cases $m=1$ and $m=n$ also follow directly from the definition.
\end{proof}

\begin{proof}[Proof of \cref{prop:resolvent}]
The function $b$ is identically 0 for $m=1$ and for $m=n$, so assume $2\leq m\leq n-1$. Set
\[
\rho=\frac mn,
\qquad
\alpha=\frac{m-1}{n-1},
\qquad
\eta_z=\1_{\{z\in R\}},
\qquad
D_z=\1_{\{X=Y=z\}}.
\]
Then \eqref{eq:source} can be rewritten as
\begin{equation}\label{eq:edge-source}
b
=
\sum_{\{i,j\}\in E(\mathcal C_n)}
\left[D_i(\eta_j-\alpha)+D_j(\eta_i-\alpha)\right].
\end{equation}

A \dfn{matching} in $\mathcal C_n$ is a set of edges of $\mathcal C_n$, no two of which share a vertex. The edge set of $\mathcal C_n$ can be partitioned into three matchings. Fix one matching $M$ from such a partition, and define
\begin{equation}\label{eq:matching-source}
g_M
=
\sum_{\{i,j\}\in M}
\left[D_i(\eta_j-\alpha)+D_j(\eta_i-\alpha)\right].
\end{equation}
Since $b$ is the sum of three functions of the form $g_M$, the triangle inequality for the norm $\|\cdot\|_{-1}$ shows that it is enough to prove the inequality 
\begin{equation}\label{eq:matching-target}
\|g_M\|_{-1}^2
\leq
K_4\frac{1-\rho}{m}(1+\log n)^2
\end{equation}
with an absolute constant $K_4$ independent of $M$.

Choose a cycle edge $e_\star\notin M$, and delete it. The remaining graph is a path containing every edge of $M$. Along this path, regard the two endpoints of each edge of $M$ as a 2-vertex atom, and regard every vertex not incident to an edge of $M$ as a 1-vertex atom. These atoms form a partition of the path into consecutive blocks.

We construct a rooted binary tree whose nodes are intervals of this path. The root is the entire path. If a node contains more than one atom, split it at an edge at the boundary of an atom whose distance from the midpoint of the node is minimal; its two children are the intervals on the two sides of that boundary edge. Continue until each leaf is a single atom. (See \cref{fig:tree}.) No edge of $M$ is cut by any split. Each child contains at least one quarter as many vertices as its parent, so the number of levels is at most $K_5(1+\log n)$ for an absolute constant $K_5$. Nodes on the same level are pairwise disjoint intervals.

\begin{figure}[ht]
\begin{center}{\includegraphics[width=\linewidth]{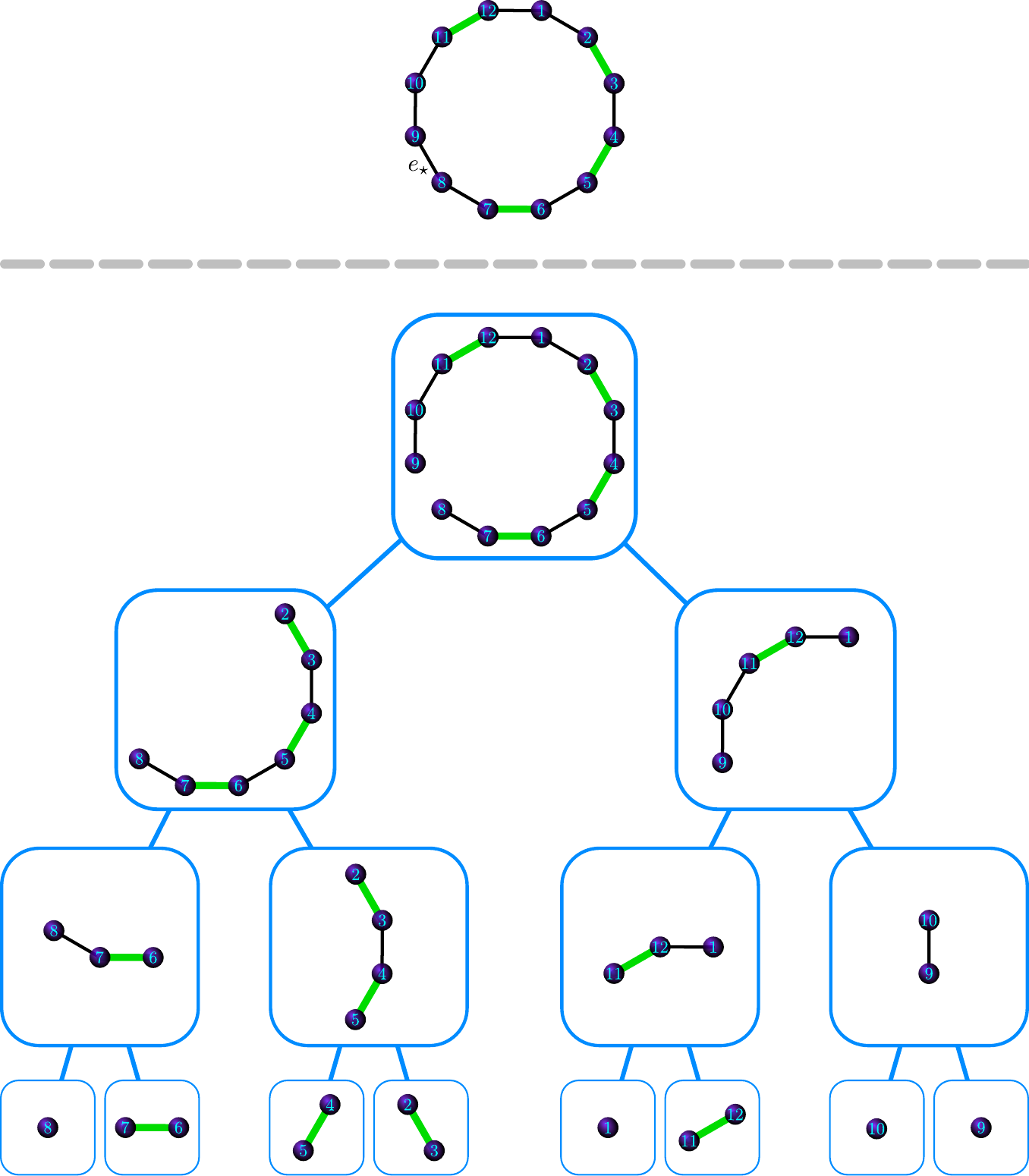}}
\end{center}
\caption{On the top is the cycle $\mathcal C_{12}$ with a specified edge $e_\star$ and a matching $M$ consisting of thick green edges. On the bottom is a binary tree obtained by repeatedly cutting edges not in $M$. }   
\label{fig:tree}
\end{figure} 

Fix a node $B$ of this tree. Let
\[
s=|B|,
\qquad
K_B=|R\cap B|,
\qquad
c_B=2|\{e\in M:e\subseteq B\}|.
\]
Thus, $c_B$ is the number of oriented summands in \eqref{eq:matching-source} contributed by matching edges contained in $B$. Let $g_{M,B}$ denote the sum of those summands.

Define $\mathcal F_B$ to be the sigma-algebra generated by $R\cap B^c$ together with the following information about each auxiliary coordinate: if the coordinate lies outside $B$, record its exact position; if it lies in $B$, record only that fact. The value of $K_B$ is $\mathcal F_B$-measurable because $|R|=m$. Let
\[
A_B=\{X\in B,\ Y\in B\}.
\]
The event $A_B$ is also $\mathcal F_B$-measurable. Conditional on $\mathcal F_B$ and on $A_B$, if $K_B=k\geq1$, the set $R\cap B$ is uniformly distributed among the $k$-element subsets of $B$, and conditional on $R\cap B$, the coordinates $X$ and $Y$ are independent and uniform on $R\cap B$. Thus, after identifying the interval $B$ with the path $[s]$, the conditional law of the internal state is the uniform measure $\nu_{s,k}$ from \cref{sec:local-gap}.

Define $F_B=\E_{\varpi_{n,m}}[g_{M,B}\mid\mathcal F_B]$.
If $s\geq 2$, then
\begin{equation}\label{eq:block-mean}
F_B
=
\1_{A_B}\frac{c_B}{sK_B}
\left(\frac{K_B-1}{s-1}-\alpha\right),
\end{equation}
where the right-hand side is defined to be $0$ when $K_B=0$. For a singleton node, set $F_B=0$.

To verify \eqref{eq:block-mean}, condition on $\mathcal F_B$, $A_B$, and $K_B=k\geq1$. For distinct vertices $i,j\in B$, the conditional uniformity described above yields the identities
\[
\E[D_i\mid\mathcal F_B,A_B]
=
\frac1{sk}
\quad\text{and}\quad \E[D_i\eta_j\mid\mathcal F_B,A_B]
=
\frac{k-1}{s(s-1)k}.
\]
There are $c_B$ oriented summands in $g_{M,B}$, so summing these conditional expectations yields \eqref{eq:block-mean}. If $A_B$ does not occur, then $g_{M,B}=0$, which is also consistent with \eqref{eq:block-mean}.

At the root $B=\ZZ/n\ZZ$, we have $K_B=m$, so $(K_B-1)/(n-1)=\alpha$, which implies that $F_B=0$. If an internal node $B$ has children $B_1$ and $B_2$, define
\begin{equation}\label{eq:UB-definition}
U_B=F_{B_1}+F_{B_2}-F_B.
\end{equation}
Because no matching edge is cut by the split, we know that 
\[
g_{M,B}=g_{M,B_1}+g_{M,B_2}.
\]
Moreover, $\mathcal F_B\subseteq\mathcal F_{B_i}$ for $i=1,2$. The tower property therefore tells us that $\E[U_B\mid\mathcal F_B]=0$. Recursively expanding $g_{M,B}-F_B$ from the root to the leaves yields the identity
\begin{equation}\label{eq:tree-telescope}
g_M
=
\sum_{B\text{ a leaf}}(g_{M,B}-F_B)
+
\sum_{B\text{ internal}}U_B.
\end{equation}
Each summand in \eqref{eq:tree-telescope} has $\varpi_{n,m}$-mean $0$.

We first bound the second moment of $F_B$. Under $\varpi_{n,m}$, conditional on $R$, the coordinates $X$ and $Y$ are independent and uniform on $R$. Hence,
\[
\Prob(A_B\mid R)=\frac{K_B^2}{m^2}.
\]
Squaring \eqref{eq:block-mean} and then taking conditional expectation given $R$ cancels the factor $K_B^{-2}$. Since $R$ is a uniformly random $m$-element subset under $\varpi_{n,m}$, the variable $K_B$ has the hypergeometric distribution with parameters $(n,m,s)$. Therefore,
\begin{equation}\label{eq:block-second-moment}
\begin{split}
\E_{\varpi_{n,m}}[F_B^2]
&=
\frac{(c_B/s)^2}{m^2}
\E\left[
\left(\frac{K_B-1}{s-1}-\alpha\right)^2
\1_{\{K_B\geq1\}}
\right]\\
&\leq
\frac{K_6\rho(1-\rho)}{m^2s},
\end{split}
\end{equation}
where we used the inequality $c_B\leq s$ and \cref{lem:hypergeometric}. Here and below, $K_6,K_7,\ldots$ denote absolute positive constants.

We next define the Dirichlet form localized to $B$. Let $E_B$ be the set of edges of the cut-open path with both endpoints in $B$. In the full Dirichlet form $\mathcal D_{n,m}(f)$, retain only the transition terms generated by spatial edges in $E_B$; denote their sum by $\mathcal D_B(f)$. The deleted edge $e_\star$ is not included, even when $B$ is the root. If we condition on $\mathcal F_B$, on $A_B$, and on $K_B=k\geq1$, then the transitions generated by $E_B$ are exactly the transitions of the path two-copy process on $s$ vertices with $k$ red vertices. Therefore, \cref{lem:local-gap} applies to the conditional internal state.

The random variable $U_B$ vanishes outside $A_B$ and satisfies $\E[U_B\mid\mathcal F_B]=0$. Consequently,
\begin{align}
|\langle f,U_B\rangle_{\varpi_{n,m}}|
&=
\left|
\E_{\varpi_{n,m}}
\left[
\bigl(f-\E[f\mid\mathcal F_B]\bigr)U_B
\right]
\right|\notag\\
&\leq
\bigl(\E_{\varpi_{n,m}}[U_B^2]\bigr)^{1/2}
\bigl(\E_{\varpi_{n,m}}[\1_{A_B}\Var(f\mid\mathcal F_B)]\bigr)^{1/2}.
\label{eq:block-cs}
\end{align}
Applying \cref{lem:local-gap} on every conditional fiber for which $A_B$ occurs, and then averaging over the fibers, yields
\[
\E_{\varpi_{n,m}}[\1_{A_B}\Var(f\mid\mathcal F_B)]
\leq
K_1|B|^2\mathcal D_B(f).
\]
Substituting this inequality into \eqref{eq:block-cs} yields
\begin{equation}\label{eq:single-block-dual}
|\langle f,U_B\rangle_{\varpi_{n,m}}|
\leq
\left(K_1|B|^2\E_{\varpi_{n,m}}[U_B^2]\right)^{1/2}
\mathcal D_B(f)^{1/2}.
\end{equation}

Let $\mathcal B$ be a collection of pairwise disjoint internal nodes. Their internal path-edge sets are pairwise disjoint, so
\[
\sum_{B\in\mathcal B}\mathcal D_B(f)
\leq
\mathcal D_{n,m}(f).
\]
Summing \eqref{eq:single-block-dual} over $B\in\mathcal B$ and applying the Cauchy--Schwarz inequality yields
\begin{equation}\label{eq:block-dual}
\left|
\left\langle f,\sum_{B\in\mathcal B}U_B\right\rangle_{\varpi_{n,m}}
\right|
\leq
\left(
K_1\sum_{B\in\mathcal B}|B|^2\E_{\varpi_{n,m}}[U_B^2]
\right)^{1/2}
\mathcal D_{n,m}(f)^{1/2}.
\end{equation}

If $B$ has children $B_1,B_2$, then the tower property yields
\[
F_B
=
\E[F_{B_1}+F_{B_2}\mid\mathcal F_B].
\]
Hence, $U_B=(F_{B_1}+F_{B_2})-\E[F_{B_1}+F_{B_2}\mid\mathcal F_B]$. Conditional expectation minimizes mean squared error among $\mathcal F_B$-measurable functions, so
\[
\E[U_B^2]
\leq
\E[(F_{B_1}+F_{B_2})^2].
\]
The product $F_{B_1}F_{B_2}$ is identically 0 because the event defining a nonzero value of either variable requires both auxiliary coordinates to lie in the corresponding child, and the two children are disjoint. Using \eqref{eq:block-second-moment} and the fact that each child has at least $|B|/4$ vertices yields
\begin{equation}\label{eq:UB-second-moment}
\E_{\varpi_{n,m}}[U_B^2]
\leq
\frac{K_7\rho(1-\rho)}{m^2|B|}.
\end{equation}

Fix one level of the binary tree, and let $\mathcal B$ be the internal nodes on that level. The nodes are pairwise disjoint and satisfy
\[
\sum_{B\in\mathcal B}|B|\leq n.
\]
Combining \eqref{eq:block-dual} and \eqref{eq:UB-second-moment} yields that 
\[
\left|
\left\langle f,\sum_{B\in\mathcal B}U_B\right\rangle_{\varpi_{n,m}}
\right|
\leq
\left(K_8\frac{n\rho(1-\rho)}{m^2}\right)^{1/2}
\mathcal D_{n,m}(f)^{1/2}.
\]
Since $n\rho=m$, the coefficient inside the square root equals $K_8(1-\rho)/m$. The variational identity \eqref{eq:variational} therefore implies that 
\begin{equation}\label{eq:1-scale}
\left\|
\sum_{B\in\mathcal B}U_B
\right\|_{-1}^2
\leq
K_8\frac{1-\rho}{m}.
\end{equation}
Here we used the fact, which is a direct consequence of \eqref{eq:variational}, that if a mean-0 function $g$ satisfies
\[
|\langle f,g\rangle_{\varpi_{n,m}}|
\leq
A^{1/2}\mathcal D_{n,m}(f)^{1/2}
\]
for every $f$, then $\|g\|_{-1}^2\leq A$.

It remains to bound the leaf terms in \eqref{eq:tree-telescope}. A 1-vertex leaf contributes $0$. A two-vertex leaf is an edge $\{i,j\}\in M$, and its contribution before centering is
\[
D_i(\eta_j-\alpha)+D_j(\eta_i-\alpha).
\]
Since the events $D_i=1$ and $D_j=1$ are disjoint, and conditional on $D_i=1$ the variable $\eta_j$ is Bernoulli with parameter $\alpha$, we obtain the identity
\begin{equation}\label{eq:leaf-second-moment}
\E_{\varpi_{n,m}}
\left[
\bigl(D_i(\eta_j-\alpha)+D_j(\eta_i-\alpha)\bigr)^2
\right]
=
\frac{2\alpha(1-\alpha)}{nm}.
\end{equation}
Subtracting a conditional expectation cannot increase the second moment. The matching $M$ has at most $n/2$ edges, and
\[
1-\alpha=\frac{n-m}{n-1}\leq\frac32(1-\rho)
\]
for $n\geq3$. Applying the size-two instance of \cref{lem:local-gap} to the disjoint leaf blocks, exactly as in the derivation of \eqref{eq:block-dual}, and then using \eqref{eq:variational}, yields the inequality 
\begin{equation}\label{eq:leaf-hminusone}
\left\|
\sum_{B\text{ a leaf}}(g_{M,B}-F_B)
\right\|_{-1}^2
\leq
K_9\frac{1-\rho}{m}.
\end{equation}

There are at most $K_5(1+\log n)$ levels. Applying the triangle inequality for $\|\cdot\|_{-1}$ to \eqref{eq:tree-telescope}, then applying \eqref{eq:1-scale} to each level and \eqref{eq:leaf-hminusone} to the leaf contribution, yields the inequality 
\[
\|g_M\|_{-1}
\leq
K_{10}(1+\log n)\sqrt{\frac{1-\rho}{m}}.
\]
Squaring proves \eqref{eq:matching-target}. Finally, $b$ is the sum of the contributions from three matchings, so one more application of the triangle inequality for $\|\cdot\|_{-1}$ proves \eqref{eq:resolvent}.
\end{proof}

\section{Proof of Cutoff}\label{sec:finish}

We now combine the entropy comparison from \cref{sec:exposure}, the integrated identity from \cref{sec:replica}, and the multiscale estimate from \cref{sec:blocks}.

\begin{lemma}\label[lemma]{lem:total-error}
There is an absolute constant $K_{11}>0$ such that
\begin{equation}\label{eq:total-error}
\sum_{m=1}^n m\int_0^\infty\mathcal E_{n,m}(t)\,\mathrm{d}t
\leq
K_{11}n^2(1+\log n)^3.
\end{equation}
\end{lemma}

\begin{proof}
Multiply \eqref{eq:integrated-error} by $m$ and sum over $m\in[n]$. The first term on the right-hand side of \eqref{eq:integrated-error} is nonpositive, so discarding it yields an upper bound.

For the terms involving $\mathscr R_{n,m}$, \cref{prop:resolvent} yields
\begin{align*}
\sum_{m=1}^n\frac{(n-1)^2}{4}\mathscr R_{n,m}
&\leq
K_{12}n^2(1+\log n)^2
\sum_{m=1}^n\frac{1-m/n}{m}\\
&\leq
K_{13}n^2(1+\log n)^3.
\end{align*}

For the terms involving $B_n$, we have
\begin{align*}
\sum_{m=1}^n
m\left(\frac nm\right)^2\frac{n-m}{n-1}B_n
&\leq
n^2B_n\sum_{m=1}^n\frac1m\\
&\leq
K_{14}n^2(1+\log n)^2,
\end{align*}
where the last inequality follows from \eqref{eq:B}. Combining the two estimates proves \eqref{eq:total-error}.
\end{proof}

\subsection{The Upper Bound}

We first prove the upper bound in \cref{thm:main}. The estimate in \cref{lem:total-error} is integrated over time. If the same Markov transition kernel is applied to two probability measures, their relative entropy cannot increase. Since $\pi_n$ is stationary, the function $t\mapsto H(\mu_t\mid\pi_n)$ is therefore nonincreasing. We use this fact to convert a time-averaged entropy estimate into a pointwise estimate.

\begin{proof}[Proof of the upper bound in \cref{thm:main}]
For all sufficiently large $n$, every $t$ in the interval
\[
J_n=\left[t_n,t_n+\frac1{\lambda_n}\right]
\]
satisfies $t\geq a_0/\lambda_n$. Hence, \cref{lem:heat} and the identity $e^{-2\lambda_nt_n}=1/n$ yield
\[
nh(t)
\leq
K_{15}ne^{-2\lambda_nt}
\leq
K_{15}
\]
for $t\in J_n$. Integrate \eqref{eq:entropy-error} over $J_n$. Since $\mathcal E_{n,m}(t)\geq0$, replacing the integral over $J_n$ by the integral over $[0,\infty)$ only increases the error term. Using \cref{lem:total-error} and the inequality
\[
\lambda_nn^2
=4n^2\sin^2(\pi/n)
\leq4\pi^2,
\]
we obtain
\begin{equation}\label{eq:averaged-entropy}
\lambda_n\int_{J_n}H(\mu_t\mid\pi_n)\,\mathrm{d}t
\leq
K_{16}(1+\log n)^3.
\end{equation}

Set
\[
T_n=t_n+\frac1{\lambda_n}.
\]
The function $t\mapsto H(\mu_t\mid\pi_n)$ is nonincreasing. Since $J_n$ has length $1/\lambda_n$, \eqref{eq:averaged-entropy} yields the inequality
\begin{equation}\label{eq:entropy-at-T}
H(\mu_{T_n}\mid\pi_n)
\leq
K_{16}(1+\log n)^3.
\end{equation}

Apply the estimate \eqref{eq:entropy-decay} starting at time $T_n$. For every $u\geq 0$,
\begin{equation}\label{eq:post-T-decay}
H\left(\mu_{T_n+u/\lambda_n}\mid\pi_n\right)
\leq
K_{16}(1+\log n)^3e^{-c_5u},
\end{equation}
where $c_5>0$ is the absolute constant in \eqref{eq:entropy-decay}. Choose
$u=K_{17}\log\log n+K_\varepsilon$,
where $K_{17}$ is a sufficiently large absolute constant and $K_\varepsilon$ is sufficiently large depending only on $\varepsilon$. Then \eqref{eq:post-T-decay} and Pinsker's inequality imply that 
\[
d_n\left(T_n+\frac{u}{\lambda_n}\right)\leq\varepsilon
\]
for all sufficiently large $n$. Since $\lambda_n^{-1}\asymp n^2$, there is an absolute constant $K_{18}>0$ such that
\[
t_{\mathrm{mix}}^{(n)}(\varepsilon)
\leq
t_n+K_{18}n^2\log\log n+K'_{\varepsilon}n^2
\]
for a constant $K'_{\varepsilon}>0$ depending only on $\varepsilon$. This is the upper bound in \cref{thm:main} after renaming the constants appearing there.
\end{proof}

\subsection{The Lower Bound}

We prove the lower bound using the statistic $F$ defined in \cref{sec:intro}. The argument is the eigenfunction method used by Wilson \cite{Wilson}.

\begin{proof}[Proof of the lower bound in \cref{thm:main}]
Recall that
\[
F(\sigma)
=
\sum_{i\in\ZZ/n\ZZ}
\cos\left(\frac{2\pi(\sigma(i)-i)}n\right).
\]
For a fixed label $i$, the function
\[
\sigma\longmapsto
\cos\left(\frac{2\pi(\sigma(i)-i)}n\right)
\]
is a first Fourier mode of the position of card $i$. Therefore,
\begin{equation}\label{eq:F-eigenfunction}
\mathcal L_nF=-\lambda_nF.
\end{equation}
Since $F(\id)=n$, the function $m(t)=\E_{\mu_t}[F]$ satisfies $m'(t)=-\lambda_nm(t)$ and $m(0)=n$. Solving this differential equation yields the identity
\begin{equation}\label{eq:F-mean}
\E_{\mu_t}[F]
=
ne^{-\lambda_nt}.
\end{equation}
Symmetry of the uniform measure yields $\E_{\pi_n}[F]=0$.

If one adjacent transposition occurs, exactly two cards move by one cyclic step. Since the function $\theta\mapsto\cos\theta$ is $1$-Lipschitz and one cyclic step changes the angle by $2\pi/n$, we have
\[
|F(\tau_x\circ\sigma)-F(\sigma)|
\leq
\frac{4\pi}{n}
\]
for all $x$ and $\sigma$. Consequently,
\begin{equation}\label{eq:F-carré}
\sum_{x\in\ZZ/n\ZZ}
\bigl(F(\tau_x\circ\sigma)-F(\sigma)\bigr)^2
\leq
\frac{16\pi^2}{n}.
\end{equation}

For any function $G$ on the state space of a continuous-time Markov chain,
\[
\mathcal L_n(G^2)-2G\mathcal L_nG
=
\sum_x\bigl(G(\tau_x\circ\sigma)-G(\sigma)\bigr)^2.
\]
Applying this identity with $G=F$ and using \eqref{eq:F-eigenfunction}, we find that 
\begin{equation}\label{eq:F-variance-derivative}
\frac{\mathrm{d}}{\mathrm{d}t}\Var_{\mu_t}(F)
=
-2\lambda_n\Var_{\mu_t}(F)
+
\E_{\mu_t}\left[
\sum_x\bigl(F(\tau_x\circ\sigma)-F(\sigma)\bigr)^2
\right].
\end{equation}
Because $\Var_{\mu_0}(F)=0$, the differential inequality obtained from \eqref{eq:F-carré} and \eqref{eq:F-variance-derivative} yields
\begin{equation}\label{eq:F-variance-bound}
\Var_{\mu_t}(F)\leq K_{19}n
\end{equation}
for every $t\geq0$, where we used the fact that $\lambda_n\asymp n^{-2}$. 

For completeness, we compute the variance under stationarity exactly. Define
\[
a_{i,x}=\cos\left(\frac{2\pi(x-i)}n\right).
\]
Every row sum and every column sum of the matrix $(a_{i,x})$ is $0$, and
\[
\sum_{i,x}a_{i,x}^2=\frac{n^2}{2}.
\]
For a uniform random permutation $\sigma$, we can expand
\[
F(\sigma)^2
=
\left(\sum_i a_{i,\sigma(i)}\right)^2
\]
and use the identities 
\[
\Prob(\sigma(i)=x)=\frac{1}{n}\quad\text{and}\quad\Prob(\sigma(i)=x,\sigma(j)=y)=\frac1{n(n-1)}
\]
for $i\neq j$ and $x\neq y$. Using the zero row and column sums, we compute that 
\begin{equation}\label{eq:F-stationary-variance}
\Var_{\pi_n}(F)
=
\frac1{n-1}\sum_{i,x}a_{i,x}^2
=
\frac{n^2}{2(n-1)}.
\end{equation}

Fix $s\geq0$ and set
$t=t_n-s/\lambda_n$.
For each fixed $s$, this time is nonnegative for all sufficiently large $n$. By \eqref{eq:F-mean} and the definition of $t_n$, we have $\E_{\mu_t}[F]
=
\sqrt n\,e^s$.
Let
\[
A_s=\left\{\sigma:F(\sigma)\geq\frac12\sqrt n\,e^s\right\}.
\]
Chebyshev's inequality and \eqref{eq:F-variance-bound} imply that 
$\mu_t(A_s)
\geq
1-K_{20}e^{-2s}$,
while Chebyshev's inequality and \eqref{eq:F-stationary-variance} imply that 
$\pi_n(A_s)
\leq
K_{21}e^{-2s}$.
Hence, 
\begin{equation}\label{eq:lower-tv}
d_n\left(t_n-\frac{s}{\lambda_n}\right)
\geq
1-K_{22}e^{-2s}.
\end{equation}

Given $\varepsilon\in(0,1)$, choose $s_\varepsilon>0$ so that the right-hand side of \eqref{eq:lower-tv} is greater than $\varepsilon$. Since $\lambda_n^{-1}\asymp n^2$, there is a constant $K''_\varepsilon>0$, depending only on $\varepsilon$, such that
\[
t_{\mathrm{mix}}^{(n)}(\varepsilon)
\geq
t_n-K''_\varepsilon n^2
\]
for all sufficiently large $n$. This is the lower bound in \cref{thm:main}.

Finally, we have 
\[
\lambda_n\sim\frac{4\pi^2}{n^2}
\qquad\text{and}\qquad
\log\log n=o(\log n).
\]
Combining the upper and lower bounds therefore yields the fact that 
\[
t_{\mathrm{mix}}^{(n)}(\varepsilon)
\sim
\frac{n^2\log n}{8\pi^2}
\]
for every fixed $\varepsilon\in(0,1)$, which proves the cutoff assertion in \cref{thm:main}.
\end{proof}

\section{Further Directions}\label{sec:future}

\subsection{The Width of the Transition and a Limiting Profile}

The bounds in \cref{thm:main} tell us that
\[
t_n-C_\varepsilon n^2
\leq
t_{\mathrm{mix}}^{(n)}(\varepsilon)
\leq
t_n+Cn^2\log\log n+C_\varepsilon n^2
\]
for every fixed $\varepsilon\in(0,1)$. 

\begin{question}\label[question]{ques:window}
Does
\[
t_{\mathrm{mix}}^{(n)}(\varepsilon)=t_n+O_\varepsilon(n^2)
\]
hold for every fixed $\varepsilon\in(0,1)$? For every fixed $s\in\RR$, does the limit
\[
\lim_{n\to\infty}d_n(t_n+sn^2)
\]
exist? If it exists for all $s$, what is the resulting function of $s$?
\end{question}

We note that Lacoin \cite{LacoinProfile} determined an analogous limiting total-variation profile for the simple exclusion process on the cycle.   

\subsection{Other Graphs}

The random-exposure argument from \cref{sec:exposure}, the two-copy
construction from \cref{sec:replica}, and the second-moment identity
used in the proof of \cref{lem:error-identity} all have analogues on
arbitrary finite connected graphs. However, extending the proof of cutoff requires further work. \cref{prop:integrated} relies on the cyclic separation $X-Y$ and the explicit function in \eqref{eq:corrector}, while the estimate in \cref{prop:resolvent} uses a decomposition into nested intervals and the inequality in \cref{lem:local-gap}. It would be interesting to identify other graph classes for which analogues of these two estimates can be obtained. 

\section{Appendix: A Formal Certificate in Lean}\label{sec:formalization}

The proofs in this paper were generated through human-AI collaboration. In
dialogue with AI, the author developed and formalized \cref{thm:main} with
AxiomCode, an AI system currently under development by Axiom Math. In
particular, this resulted in a formal Lean certificate for \cref{thm:main} built on the
Mathlib library. The formalization does not assume any results from the
literature beyond those already verified in Mathlib. In particular, the
logarithmic Sobolev inequality of Lee and Yau \cite{LeeYau} used in the proof
of \cref{lem:lsi}, as well as Pinsker's inequality, are proved within the
formalization. See
\begin{center}
\url{https://github.com/AxiomMath/CycleCutoff}
\end{center}
for explicit details. This repository includes a formal challenge file containing the statement of
\cref{thm:main}, which can be mechanically verified using the Comparator tool in
Lean.

\end{document}